\documentclass{article}

\usepackage{latexsym,epsfig,tabularx,amssymb,enumerate}

\usepackage{microtype}
\usepackage[title,titletoc]{appendix}

\newcommand{\showkeyslabelformat}[1]{%
}
\usepackage[notref,notcite]{showkeys}

\usepackage{graphicx}
\usepackage{color}
\usepackage{bbm}
\usepackage{amsmath,amsfonts,amsthm,bm}
\usepackage[a4paper]{geometry}

\newcommand{\Z}{\mathbb{Z}} 
\newcommand{\C}{\mathbb{C}} 
\newcommand{\R}{\mathbb{R}} 
\newcommand{\N}{\mathbb{N}} 
\newcommand{\M}{\mathbb{M}}

\newcommand{\imagunit}{\mathrm{i}}
\newcommand{\twopii}{2 \pi \imagunit \,}

\newcommand{\shinv}{\mathrm{shinv}}

\DeclareSymbolFont{bbold}{U}{bbold}{m}{n}
\DeclareSymbolFontAlphabet{\mathbbold}{bbold}
\newcommand{\ind}[1]{\mathbbm{1}_{#1}}

\newcommand*{\link}[1]{(\ref{#1})}                                      
\newcommand*{\abs}[1]{\left| #1 \right|}                                
\newcommand*{\nach}{\rightarrow}                                        
\newcommand*{\norm}[1]{\left\| #1 \right\|}                             
\newcommand*{\sep}{\; \vrule \;}                                       
\newcommand{\ie}{i.e.}
\renewcommand*{\S}{\mathcal{S}}                                         
\renewcommand*{\P}{\mathcal{P}}                                         
\newcommand*{\floor}[1]{\left\lfloor #1 \right\rfloor}                  
\newcommand*{\distr}[2]{\left\langle #1, #2 \right\rangle}              
\newcommand{\eg}{e.g.}
\newcommand*{\SI}{\mathfrak{S}}                                         
\newcommand*{\id}{\mathrm{id}}                                          
\newcommand*{\0}{\mathcal{O}}                                           

\usepackage[draft=false,colorlinks,bookmarksnumbered,linkcolor=black, citecolor=black]{hyperref}

\usepackage{aliascnt}

\theoremstyle{plain}
\newtheorem{theorem}{Theorem}[section]

\newaliascnt{lem}{theorem}
\newtheorem{lemma}[lem]{Lemma}

\aliascntresetthe{lem}
\newaliascnt{prop}{theorem}
\newtheorem{prop}[prop]{Proposition}

\aliascntresetthe{prop}
\newaliascnt{cor}{theorem}
\newtheorem{corollary}[cor]{Corollary}

\aliascntresetthe{cor}

\theoremstyle{definition}

\newaliascnt{ex}{theorem}

\aliascntresetthe{ex}    

\newaliascnt{ass}{theorem}
\newtheorem{assumption}[ass]{Assumption}

\aliascntresetthe{ass} 

\newaliascnt{defi}{theorem}

\aliascntresetthe{defi} 

\theoremstyle{remark}

\newaliascnt{rem}{theorem}
\newtheorem{remark}[rem]{Remark}

\aliascntresetthe{rem}
      
\newcommand{\setu}{\mathfrak{u}}
\newcommand{\setv}{\mathfrak{v}}

\newcommand{\rd}{\,\mathrm{d}} 

\newcommand{\bszero}{\boldsymbol{0}} 
\newcommand{\bst}{\boldsymbol{t}}    
\newcommand{\bsh}{\boldsymbol{h}}    
\newcommand{\bsk}{\boldsymbol{k}}    
\newcommand{\bsx}{\boldsymbol{x}}    
\newcommand{\bsy}{\boldsymbol{y}}    
\newcommand{\bsz}{\boldsymbol{z}}    
\newcommand{\bsDelta}{\boldsymbol{\Delta}}    

\newcommand{\tpmod}[1]{~(\operatorname{mod}{#1})} 

\newcommand{\tr}[1]{#1}
\usepackage{soul}

\def\citep#1#2{\cite[{#1}]{#2}}

\newcommand{\ns}{\negthickspace\negthickspace}
\newcommand{\NR}[1]{\mu_R(#1)}
\newcommand{\NN}{\eta}

\begin{document}

\title{Construction of quasi-Monte Carlo rules for multivariate integration in spaces of permutation-invariant functions\footnote{This work has been supported by Deutsche Forschungsgemeinschaft DFG (DA 360/19-1).}}
\author{
Dirk Nuyens\thanks{Department of Computer Science, KU Leuven, Celestijnenlaan 200A, 3001 Leuven, Belgium.}~\thanks{Email: dirk.nuyens@cs.kuleuven.be}
\and Gowri Suryanarayana\footnotemark[2]~\thanks{VITO, Energy Technology, Boeretang 200, B-2400 Mol, Belgium. Corresponding author. Email: gowri.suryanarayana@vito.be}
\and Markus Weimar\thanks{Philipps-University Marburg, Faculty of Mathematics and Computer Science, Hans-Meerwein-Stra{\ss}e, Lahnberge, 35032 Marburg, Germany. 
Current address: University of Siegen, Department of Mathematics, Walter-Flex-Stra{\ss}e 3, 57068 Siegen, Germany. Email: weimar@mathematik.uni-marburg.de}
}

\maketitle

\begin{abstract}
\noindent We study multivariate integration of functions that are invariant under the permutation (of a subset) of their arguments.
Recently, in \tr{Nuyens, Suryanarayana, and Weimar \cite{NSW14} (\textit{Adv.\ Comput.\ Math.} (2016), 42(1):55--84), the authors derived} an upper estimate for the $n$th minimal worst case error for such problems,  and show\tr{ed} that under certain conditions this upper bound only weakly depends on the \tr{dimension}. We extend these results by proposing two (semi-) explicit construction schemes.
We develop a component-by-component algorithm to find the generating vector for
a shifted rank-$1$ lattice rule that obtains a rate of convergence arbitrarily close to $\0(n^{-\alpha})$, where $\alpha>1/2$ denotes the smoothness of our function space \tr{and $n$ is the number of cubature nodes}. 
Further, we develop a semi-constructive algorithm that 
builds on point sets which can be used to approximate the integrands of interest with a small error; 
the 
cubature error is then bounded by the error 
of approximation. Here the same rate of convergence is achieved while the dependence of the error bounds on the dimension $d$ is significantly improved.


\end{abstract}

\section{Introduction}\label{sect:Intro}
In recent times, the efficient calculation of multidimensional integrals has become more and more important, especially when working with a very high number of dimensions. It is well-known that such problems are numerically feasible only if certain additional assumptions on the integrands of interest are imposed.
In this paper we seek to construct cubature rules to approximate integrals of $d$-variate functions which are invariant under the permutation of (subsets of) their arguments. 
Such a setting was studied earlier in \cite{W12} in the context of approximations to general linear operators and also recently in \cite{NSW14, W14} 
for the special case of integration. 
These investigations were motivated by recent research by Yserentant~\cite{Y10}, who proved 
that the rate of convergence of the approximations to the electronic Schr\"{o}dinger equation of $N$-electron systems does not depend on the number of electrons $N$, thus showing that it is independent of the number of \tr{variables}
associated with the numerical problem.
This is due to the fact that there is some inherent symmetry in the exact solutions (called electronic wave functions) to these equations. In fact, such functions are antisymmetric w.r.t.\ the exchange of electrons having the same spin, as described by \emph{Pauli's exclusion principle} which is well-known in theoretical physics.
In \cite{W12}, general linear problems defined on spaces equipped with this type of (anti-) symmetry constraints were shown to be (strongly) polynomially tractable (under certain conditions). That is, the complexity of solving such problems  depends at most weakly on the \tr{dimension}. 
Inspired by these results, it was shown in \cite{NSW14} that under moderate assumptions, the $n$th minimal worst case error for integration in the permutation-invariant (hence symmetric) setting of certain periodic function classes can also be bounded nicely.
In particular, the existence of quasi-Monte Carlo (QMC) rules that can attain the rate of convergence $\0(n^{-1/2})$ (known from Monte Carlo methods) was established. 
If there are enough \tr{variables} that are equivalent with respect to permutations, then the implied constant in this bound grows only polynomially with the dimension~$d$. Hence, the integration problem was shown to be polynomially tractable. However, the proofs given in \cite{NSW14} are based on non-constructive averaging arguments.

In this paper, we extend \tr{the results of \cite{NSW14}} and provide two (semi-) explicit construction schemes.
First, we present a procedure which is inspired by the well-known \emph{component-by-component} (CBC) search algorithm for lattice rules, introduced by Sloan and Reztsov~\cite{SR02}.
More precisely, we prove that the generating vector of a randomly shifted rank-$1$ lattice rule which achieves an order of convergence arbitrarily close to $\0(n^{-\alpha})$ can be found component-by-component.
 Here $\alpha>1/2$ denotes the 
smoothness of our permutation-invariant integrands and $n$ is the number of cubature nodes from a $d$-dimensional integration lattice.
Originally\tr{,} the CBC algorithm was developed with the aim of reducing the search space for the generating vector, i.e., to make the search quicker. 
Later on it was shown by Kuo~\cite{K03} that this algorithm can also be used to find lattice rules \tr{that} achieve optimal rates of convergence in weighted function spaces of Korobov and Sobolev type\tr{, and this} improved on the Monte Carlo rate $\0(n^{-1/2})$ shown previously in \cite{JK02,SKJ02} and \cite{SKJ02_1}.
For an overview of these schemes for different problem settings, we refer to \cite{DKS13,Nuy2014}.

Next, we present a completely different approach towards cubature rules that does not rely on the concept of lattice points.
Based on the knowledge about sampling points that are suitable for approximation, 
we iteratively derive a cubature rule which attains the desired order of convergence for solving the integration problem in more general reproducing kernel Hilbert spaces (RKHSs).
The basic idea of this type of algorithms can be traced back to Wasilkowski~\cite{W94}.
The proof relies on the fact that in RKHSs the worst case error for integration can be controlled by \tr{an average case error for the function approximation problem. }
This way we arrive at a cubature rule which achieves a rate of  convergence arbitrarily close to the optimal rate of $\0(n^{-\alpha})$, and prove that under certain conditions, we can achieve (strong) polynomial tractability.

Our material is organized as follows. 
In \autoref{sect:setting} we briefly recall the problem setting discussed in~\cite{NSW14}. 
In particular, we define the 
permutation-invariant subspaces of the periodic RKHSs under consideration. 
\autoref{sect:avgbounds} introduces the worst case error of the problem and \autoref{thm:tractability} outlines the main tractability results derived in 
\cite{NSW14}.  \autoref{sect:cbc} and \autoref{sect:alternate_construction} contain our main results, namely the new CBC construction and the application of the second cubature rule explained to permutation invariant subspaces, respectively. 
In \autoref{sect:cbc} we start with a brief description of (shifted) rank-$1$ lattice rules and recall some related results obtained in \cite{NSW14}. 
Then, \autoref{thm:cbc} in \autoref{subsect:cbc} contains our new CBC construction.
\autoref{sect:alternate_construction} is devoted to the alternate approach.
We begin with a short overview of basic facts related to approximation problems w.r.t.\ the average case setting in \autoref{subsect:approx}.
Afterwards, in \autoref{subsect:quad}\tr{,} our quadrature rule for general RKHSs is derived.
The corresponding error analysis can be found in \autoref{thm:general_quad}.
Finally, in \autoref{subsect:korobov}, the obtained results are applied to the permutation-invariant setting described in \autoref{sect:setting}.
The paper is concluded with an appendix which contains the proofs of
some technical lemmas needed in our derivation.

\tr{We briefly introduce some notation that will be used throughout the paper. Vectors are denoted by bold face letters, \eg, $\bsx,\bsk,\bsh$. In contrast, scalars  are denoted by, \eg, $x,k,h$. To denote multidimensional spectral indices we mainly use $\bsh$ or $\bsk$.  The elements of a $d$-dimensional vector are denoted as $\bsk=(k_1,\ldots,k_d)$. We use $\bsx\cdot\bsy$ to denote the inner product or the dot product of two vectors $\bsx$ and $\bsy$. The prefix $\#$ is used for the cardinality of a set. Finally, the norm of a function $f$ in a function space $H_d$ is denoted by $\norm{f\sep H_d}$.}
\section{Setting}\label{sect:setting}
We study multivariate integration 
\begin{gather}\label{Int}
		\mathrm{Int}_d f \tr{:=} \int_{[0,1]^d} f(\bsx) \rd \bsx
\end{gather}
for functions from subsets of some Hilbert space of periodic functions 
\begin{align*}
 		F_{d}(r_{\alpha,\bm{\beta}}) \tr{:=} \left\{ f:[0,1]^d \rightarrow \mathbb{C} 
 				\sep f\in L_2([0,1]^d) \text{ and } \norm{f\sep F_{d}(r_{\alpha,\bm{\beta}})}^2 \tr{:=} \sum_{\bsk \in \Z^d} \abs{\widehat{f}(\bsk)}^2 r_{\alpha,\bm{\beta}}(\bsk) < \infty \right\}.
\end{align*}
\tr{
Here $\widehat{f}(\bsk)$ are the Fourier coefficients of $f$, given by}
\begin{align*}
		\widehat{f}(\bsk) := \distr{f}{\exp(\twopii \bsk \cdot \cdot)}_{L_2} 
		= \int_{[0,1]^d} f(\bsx) \, \exp(-\twopii \bsk \cdot \bsx) \rd\bsx,
\end{align*}
\tr{and }$r_{\alpha,\bm{\beta}} \colon \Z^d \nach (0,\infty)$
is a $d$-fold tensor product involving some generating function $R\colon [1,\infty) \nach (0,\infty)$ and 
a tuple $\bm{\beta} = (\beta_0,\beta_1)$ of positive parameters
such that
\begin{gather*}
		r_{\alpha,\bm{\beta}}(\bsk) 
		= \prod_{\ell=1}^d \left( \delta_{0,k_\ell} \cdot \beta_0^{-1} + (1-\delta_{0,k_\ell}) \cdot \beta_1^{-1} \cdot R(\abs{k_\ell})^{2\alpha} \right), 
		\quad \bsk\in\Z^d.
\end{gather*}
Note that the problem is said to be well scaled if $\beta_0=1$. The parameter $\alpha \geq 0$ describes the smoothness. 
Throughout this paper we assume that 
\begin{equation*}
	\frac{1}{c_R} \,R(m) \leq \frac{R(n\,m)}{n} \leq R(m)
	\qquad \text{for all} \quad \tr{n,} m\in\N, \quad \text{and some} \quad c_R\geq 1.
\end{equation*}
Moreover, we assume that
$(R(m)^{-1})_{m\in\N} \in \ell_{2\alpha}$\tr{;} \ie,
\begin{gather*}
		\NR{\alpha}:=\sum_{m=1}^\infty \frac{1}{R(m)^{2\alpha}} < \infty.
\end{gather*}
The functions $f\in F_{d}(r_{\alpha,\bm{\beta}})$ possess an absolutely convergent Fourier expansion. 
Also, the above assumptions imply $R(m)\sim m$ and $\alpha>1/2$, and 
for these conditions, it is known that
$F_d$ is a $d$-fold tensor product 
of some univariate reproducing kernel Hilbert space (RKHS) $F_1=H(K_1)$ with kernel $K_1$, \tr{see, \eg, \cite[Appendix A]{NW08}}. 
The inner product on the RKHS $F_d = H(K_d)$ is then given by
\begin{align*}
		\distr{f}{g} = \sum_{\bsk \in \Z^d} r_{\alpha,\bm{\beta}}(\bsk) \, \widehat{f}(\bsk)\, \overline{\widehat{g}(\bsk)}.
\end{align*}

We now recall the definition of \emph{$I_d$-permutation-invariant functions} $f\in F_d(r_{\alpha,\bm{\beta}})$, where $I_d \subseteq \{1,\ldots,d\}$ is some fixed subset
of coordinates.
As discussed in \cite{W12,W14} and \cite{NSW14} these functions satisfy the constraint that they are invariant under all permutations of the variables with indices in $I_d$\tr{;} \ie,
\begin{equation}\label{def:perm-inv}
 		f(\bsx) = f(P(\bsx)) 
 		\qquad \text{for all } \quad \bsx\in[0,1]^d \quad \text{ and all } \quad P \in \S_d,
\end{equation}
where
\begin{gather*}
		\S_d \tr{:=} \S_{\{1,\ldots,d\}}(I_d) \tr{:=} \left\{ P\colon\{1,\ldots,d\}\nach\{1,\ldots,d\} \sep P \text{ a bijection such that } P\big|_{\{1,\ldots,d\}\setminus I_d} = \id \right\}.
\end{gather*}
(Note that this set always contains at least the identity permutation.) 
For example, if $I_d= \{1,2\}$, then $\S_d = \left\{(1,2,\ldots,d)\mapsto(1,2,\ldots,d),\, (1,2,\ldots,d)\mapsto(2,1,\ldots,d) \right\}$ 
and \link{def:perm-inv} reduces to the condition
$f(x_1,x_2,\ldots, x_d)= f(x_2,x_1,\ldots, x_d)$.  \tr{With slight abuse of notation, we shall use $P$ also in the functional notation; \ie, we let
\[
P(\bsx) := \left(x_{P(1)},\ldots,x_{P(d)}\right), \quad \bsx=(x_1,\ldots, x_d) \in [0,1]^d.
\]}
The subspaces of all $I_d$-permutation-invariant functions in $F_d$ will be denoted by $\SI_{I_d}(F_d(r_{\alpha,\bm{\beta}}))$.
For fully permutation-invariant functions, we will simply write $\SI(F_d(r_{\alpha,\bm{\beta}}))$.
It is known that \tr{if $I_d = \{i_1, i_2, \ldots, i_{\# I_d}\}$,} the set of all
\begin{gather*}
 		\phi_{\bsk}(\bsx) 
 		= \sqrt{\frac{r^{-1}_{\alpha,\bm{\beta}}(\bsk)}{\#\S_d \cdot \M_d(\bsk)!}} \, \sum_{P \in \S_{d}} \exp(\twopii P(\bsk)\cdot \bsx)
\end{gather*}
with \tr{$\bsk$ from the set}
\begin{gather*}
 		\nabla_d 
 		\tr{:=} \nabla_{\{1,\ldots,d\}}(I_d)
 		\tr{:=} \left\{\bsk=(k_1,\ldots,k_d) \in \Z^d \sep k_{i_1} \leq k_{i_2} \leq \tr{\cdots} \leq k_{i_{\# I_d}} \right\},
\end{gather*}
 constitutes an orthonormal basis of $\SI_{I_d}(F_d(r_{\alpha,\bm{\beta}}))$; \tr{see \cite{W12} for more details}.
Here\tr{,} 
\[
\M_d(\bsk)!\tr{:=}\M_{\{1,\ldots,d\}}(\bsk, I_d)! \tr{:=} \,\tr{\#\{ P \in \S_{d} \sep P(\bsk) = \bsk \}},
\] 
accounts for 
the repetitions in the multiindex $\bsk$. 
Since \tr{the subspace} $\SI_{I_d}(F_d(r_{\alpha,\bm{\beta}}))$ is equipped with the same norm as the entire space 
$F_d(r_{\alpha,\bm{\beta}})$, it is again a RKHS.
Moreover\tr{,} from \cite{NSW14} we know that its reproducing kernel is given by
\begin{align}
		K_{d,I_d}(\bsx,\bsy) 
		&= \sum_{\bsh \in \Z^d} \frac{r^{-1}_{\alpha,\bm{\beta}}(\bsh)}{\#\S_d} \sum_{P \in \S_d} \exp(\twopii \bsh \cdot (P(\bsx)-\bsy)),
		\qquad \bsx,\bsy\in[0,1]^d. \label{sym_kernel_2}
\end{align} 
Finally, we mention that (using a suitable rearrangement of coordinates) the space
$\SI_{I_d}(F_d(r_{\alpha,\bm{\beta}}))$ can be seen as the tensor product of
the fully permutation-invariant subset
of the $\# I_d$-variate space with the entire $(d-\# I_d)$-variate space\tr{;}
i.e.,
\begin{gather*}
		\SI_{I_d}(F_d(r_{\alpha,\bm{\beta}})) 
		= \SI(F_{\# I_d}(r_{\alpha,\bm{\beta}})) \otimes F_{d-\# I_d}(r_{\alpha,\bm{\beta}}).
\end{gather*}
Hence, the reproducing kernel also factorizes \tr{like}
\begin{gather*}
		K_{d,I_d} = K_{\# I_d, \{1,\ldots,\# I_d\}} \otimes K_{d-\#I_d}.
\end{gather*}
For more details about the setting, we refer the reader to \cite{NSW14}.

\section{Worst case error and tractability}\label{sect:avgbounds}
We approximate the integral \link{Int} by some \tr{\emph{cubature} rule} 
\begin{gather}\label{QMC}
		Q_{d,n}(f) 
		\tr{:=}\, Q_{d,n}\!\left(f; \bm{t}^{(0)},\ldots,\bm{t}^{(n-1)}, w_0,\ldots,w_{n-1}\right) 
		\tr{:=}\frac{1}{n} \sum_{j=0}^{n-1} w_j f\!\left(\bm{t}^{(j)}\right),
		\qquad d,n\in\N,
\end{gather}
that samples $f$ at some given points $\bm{t}^{(j)}\in[0,1]^d$, $j=0,\ldots,n-1$, where
the weights $w_j$ are well-chosen real numbers.
If $w_0=\tr{\cdots}=w_{n-1}=1$, then $Q_{d,n}$ is \tr{a \emph{quasi-Monte Carlo} (QMC)} rule which we will denote by $\mathrm{QMC}_{d,n} \tr{:=}\, \mathrm{QMC}_{d,n}(\,\cdot\,; \bm{t}^{(0)},\ldots,\bm{t}^{(n-1)})$.

If $K$ is the $2d$-variate, real-valued reproducing kernel of some (separable) 
RKHS $H_d$ of functions on $[0,1]^d$, the worst case error of $Q_{d,n}$ is given by
\begin{align}
		e^{\mathrm{wor}}(Q_{d,n}; H_d)^2 
		&\tr{:=} \left( \sup_{f\in H_d, \norm{f\sep H_d}\leq 1} \abs{\mathrm{Int}_d f - Q_{d,n}(f)} \right)^2 \nonumber\\
		&= \int_{[0,1]^d}\int_{[0,1]^d} K(\bsx,\bsy) \rd\bsx \rd\bsy - \frac{2}{n} \sum_{j=0}^{n-1} w_j \int_{[0,1]^d} K\!\left(\bsx,\bm{t}^{(j)}\right) \rd\bsx \nonumber\\
		&\qquad\quad + \frac{1}{n^2} \sum_{j,k=0}^{n-1} w_j w_k K\!\left(\bm{t}^{(j)},\bm{t}^{(k)} \right), \label{eq:e2-K}
\end{align}
see, e.g., Hickernell and Wo\'zniakowski \cite{IntandApp}.
The $n$th \emph{minimal worst case error} for integration on $H_d$ is then given by
\begin{gather*}
		e(n,d; H_d) \tr{:=} \inf_{Q_{d,n}} e^{\mathrm{wor}}(\tr{Q}_{d,n}; H_d).
\end{gather*}
Here the infimum is taken with respect to some class of \tr{cubature rules} $\tr{Q}_{d,n}$ which use at most 
$n$ samples of the input function. 

We briefly recall the concepts of tractability that will be used in this paper, as described in Novak and Wo{\'z}niakowski~\cite{NW08}. 
Let $n = n(\varepsilon, d)$ denote the \emph{information complexity} \tr{w.r.t.\ the \emph{normalized error criteion}, \ie,} the minimal number of function values necessary
to reduce the \emph{initial error} $e(0,d; H_d)$ by a factor of $\varepsilon > 0$, in the $d$-variate case. 
Then a problem is said to be \emph{polynomially tractable} if $n(\varepsilon ,d)$ is upper bounded by some polynomial in $\varepsilon^{-1}$
and $d$, i.e., if there exist constants $C,p>0$ and $q \geq 0$ such that for all $d\in\N$ and every $\varepsilon\in(0,1)$
\begin{align}
\label{trac_defn}
	 n(\varepsilon ,d) \le C\, d^q \, \varepsilon^{-p}.
\end{align}
If this bound is independent of $d$, i.e., if we can take $q=0$, then the problem is said to be \emph{strongly polynomially tractable}. 
Problems are called \emph{polynomially intractable} if \eqref{trac_defn} does not hold for any such choice of $C,p$, and $q$.
\tr{For the sake of completeness, we mention that a problem is said to be \emph{weakly tractable} if its information complexity does not grow exponentially with $\varepsilon^{-1}$ and $d$, \ie, if
\begin{equation*}
	\lim_{\varepsilon^{-1}+d\nach\infty} \frac{\ln n(\varepsilon,d)}{\varepsilon^{-1}+d}=0.
\end{equation*}}

In \cite[Theorem~3.6]{NSW14} the following tractability result has been shown for the spaces $H_d = \SI_{I_d}(F_d(r_{\alpha,\bm{\beta}}))$. 
\begin{prop}\label{thm:tractability}
	For $d\geq 2$\tr{,} let $I_d\subseteq\{1,\ldots,d\}$ with $\# I_d \geq 2$ and assume 
	\begin{gather}\label{assump}
 		\frac{2 \, \beta_1}{\beta_0 \, R(m)^{2\alpha}} \leq 1
 		\quad \text{for all} \quad m\in\N.
	\end{gather}
	Consider the integration problem on the $I_d$-permutation-invariant subspace $\SI_{I_d}(F_d(r_{\alpha,\bm{\beta}}))$ of $F_d(r_{\alpha,\bm{\beta}})$.
	Then
	\begin{itemize}
		\item[$\bullet$] for all $n$ and $d\in\N$, the $n$th minimal worst case error is bounded by
			\begin{align}
				e(n,d; \SI_{I_d}(F_d(r_{\alpha,\bm{\beta}}))) 
				&\leq e(0,d; \SI_{I_d}(F_d(r_{\alpha,\bm{\beta}}))) \, \sqrt{V^* + \frac{1}{1-\NN^*}}  \nonumber \\
				&\qquad\qquad\qquad \times \left(1 + \frac{2 \beta_1 \NR{\alpha}}{\beta_0} \right)^{(d-\#I_d)/2} (\#I_d)^{V^*\!/2} \frac{1}{\sqrt{n}}, \label{errorBound}
			\end{align}
			where the constants $V^*$ and $\NN^*$ are chosen such that
		 	\begin{gather}\label{boundN}
		 		\NN^* := \NN^*(V^*):= \sum_{m=V^*+1}^\infty \frac{2 \, \beta_1}{\beta_0 \, R(m)^{2\alpha}} < 1.
			 \end{gather}
		\item there exists a QMC rule which achieves this bound.
		\item if $d-\#I_d \in \0(\ln d)$, then the integration problem is polynomially tractable \tr{(with respect
to the worst case setting and the normalized error criterion)}.
		\item if $d-\#I_d \in \0(1)$ and \link{boundN} holds for $V^*=0$, then we obtain strong polynomial tractability.
		\end{itemize}		
\end{prop}

It is observed \tr{in \cite{NSW14}} that for 
the periodic unanchored Sobolev space, i.e.,
$\beta_0=\beta_1=1$ and $R(m)=2\pi m$, the assumption \link{assump} is fulfilled if $\alpha > 1/2$. 
In addition, for sufficiently many permutation-invariance conditions and sufficiently large $\alpha$, we even have strong polynomial tractability.  
Note that for $\alpha$ \tr{near 1/2}, the \tr{factor} $(1-\NN^*)^{-1/2}$ \tr{is} extremely large, whereas for $\alpha=1$ we already have $(1-\NN^*)^{-1/2} \leq 1.05$.
We stress the importance of the above tractability result by mentioning that the integration problem on the full space ($I_d = \emptyset$) is not even weakly tractable. 
That is, in this case the information complexity $n(\varepsilon,d)$ grows at least exponentially in $\varepsilon^{-1}+d$.
We refer to \cite{NSW14} for a more detailed discussion about this result.

\section{Component-by-component construction of rank-$1$ lattice rules}\label{sect:cbc}
\subsection{Definition\tr{s} and known results}
 We briefly introduce unshifted and shifted rank-$1$ lattice rules. 
 For $n\in\N$, an \emph{$n$-point $d$-dimensional rank-$1$ lattice rule} $Q_{d,n}(\bsz)$ is a QMC rule (i.e., an operator as given in \link{QMC} with $w_0=\tr{\cdots}=w_{n-1}=1$) which is fully determined by its generating vector $\bsz \in \Z_n^d:=\{0,1,\ldots,n-1\}^d$. 
 It uses points $\bst^{(j)}$ from an \emph{integration lattice} $L(\bsz,n)$ induced by~$\bsz$:
 \begin{align*}
   \bst^{(j)}
   &:=
   \left\{ \frac{\bsz j}{n} \right\}
   \tr{:=}
   \frac{\bsz j}{n} \bmod{1}
   \qquad \text{for} \qquad 
   j = 0, 1, \ldots, n-1
   .
 \end{align*}
We will restrict ourselves to prime numbers $n$ for simplicity. The following \emph{character property} over $\Z_n^d$ w.r.t.\ the trigonometric basis is useful:
\begin{equation}\label{eq:char}
  \frac1n \sum_{j=0}^{n-1} \exp(\twopii (\bsh \cdot \bsz) j / n)
  =
  \begin{cases}
    1 & \text{if $\bsh \cdot \bsz \equiv 0 \pmod{n}$}, \\
    0 & \text{otherwise}.
  \end{cases}
\end{equation}
The collection of $\bsh\in\Z^d$ for which this sum is one is called the \emph{dual lattice} and we denote it by $L(\bsz,n)^\bot$.
It has been shown in \cite[Corollary~4.7]{NSW14} that \tr{irrespective} of the \tr{set} $I_d$, for standard choices of $r_{\alpha,\bm{\beta}}$ (such as in the periodic Sobolev space or Korobov space), the class of unshifted lattice rules $Q_{d,n}(\bsz)$ is too small to obtain strong polynomial tractability.
We thus consider randomly shifted rank-$1$ lattice rules in what follows.

For 
$n\in\N$, an $n$-point $d$-dimensional \emph{shifted rank-$1$ lattice rule} consists of an unshifted rule $Q_{d,n}(\bsz)$ \tr{shifted by some shift $\bsDelta \in [0,1)^d$;}
 i.e., \tr{it's points are given by}
\begin{equation*}
  \bst^{(j)}
  := \left\{ \frac{\bsz \,j}{n} + \bsDelta \right\}
  = \left( \frac{\bsz \, j}{n} + \bsDelta \right) \bmod{1}
  \qquad \text{for} \qquad j=0,1,\ldots,n-1.
\end{equation*}
We will denote such a cubature rule by $Q_{d,n}(\bsz)+\bsDelta$.

The \emph{root mean squared worst case error} helps in establishing the existence of good shifts. 
It is given by 
\begin{equation*}
	E(Q_{d,n}(\bsz)) := \left( \int_{[0,1)^d} e^{\mathrm{wor}}(Q_{d,n}(\bsz)+\bsDelta; \SI_{I_d}(F_d(r_{\alpha,\bm{\beta}})))^2 \rd\bsDelta \right)^{1/2}.
\end{equation*}
The above error can be calculated using the \emph{shift-invariant kernel} 
\begin{equation*}
  K_{d,I_d}^{\shinv}(\bsx, \bsy)
  :=  \int_{[0,1)^d} K_{d,I_d}(\{\bsx+\bsDelta\},\{\bsy+\bsDelta\}) \rd\bsDelta,
  \qquad \bsx,\bsy\in[0,1]^d,
\end{equation*}
associated with $K_{d,I_d}$ given in~\link{sym_kernel_2}.
From \cite[Proposition~4.8]{NSW14} we know that for $d\in\N$ and $I_d\subseteq\{1,\ldots,d\}$
the shift-invariant kernel is given by
\begin{equation*}
	K_{d,I_d}^{\shinv}(\bsx, \bsy)
	= \sum_{\bsk \in \nabla_d} \frac{r_{\alpha,\bm{\beta}}^{-1}(\bsk)}{\#\S_d} \sum_{P \in S_d} \exp\left(\twopii P(\bsk) \cdot (\bsx-\bsy)\right), \qquad \bsx,\bsy\in[0,1]^d,
\end{equation*}
and, for every unshifted rank-$1$ lattice rule $Q_{d,n}(\bsz)$ the mean squared worst case error satisfies
\begin{align}
	E(Q_{d,n}(\bsz))^2 
	&= e^{\mathrm{wor}}(Q_{d,n}(\bsz); H_{d,I_d}^{\shinv})^2 \label{eq:mse}\\
	&= \sum_{\bm{0}\neq\bsk \in \nabla_d} \frac{r^{-1}_{\alpha,\bm{\beta}}(\bsk)}{\#\S_d} \sum_{P \in S_d} \ind{P(\bsk) \in L(\bsz,n)^{\perp}}
	\geq c\, \max\{d-\#I_d,1\}\, n^{-2\alpha}, \nonumber
\end{align}
where $c=2\beta_1\mu_R(\alpha)/\beta_0$ does not depend on $d$ and $n$, see \cite[Theorem~4.9]{NSW14}.
Here $H_{d,I_d}^{\shinv}$ denotes the RKHS with kernel $K_{d,I_d}^{\shinv}$ and $L(\bsz,n)^{\perp}$ is the dual lattice induced by $\bsz\in\Z_n^d$.
Moreover, the following existence result has been derived; cf.\ \cite[Theorems~4.9 and 4.11]{NSW14}.
\begin{prop}\label{prop:existence}
	Let $d\in\N$, $I_d\subseteq\{1,\ldots,d\}$\tr{,} and $n\in\N$ with $n\geq c_R$ be prime. 
	Then there exists a shifted rank-$1$ lattice rule $Q_{d,n}(\bm{z^*})+\bm{\Delta^*}$ for integration of $I_d$-permutation-invariant functions in $F_d(r_{\alpha,\bm{\beta}})$ such that
	\begin{align}
		e^{\mathrm{wor}}(Q_{d,n}(\bm{z^*})+\bm{\Delta^*}; \SI_{I_d}(F_d(r_{\alpha,\bm{\beta}})))^2  
		&
		\leq E(Q_{d,n}(\bm{z^*}))^2 \nonumber\\
		&
		\leq (1+c_R)^\lambda \, C_{d,\lambda}(r_{\alpha,\bm{\beta}}) \, \frac{1}{n^\lambda}
		\quad \text{for all} \quad
		1\leq \lambda < 2\alpha \label{est:existence}
	\end{align}
	with $c_R$ as defined in \autoref{sect:setting} and
	\begin{align}\label{C_dlambda}
		C_{d,\lambda}(r_{\alpha,\bm{\beta}})
		:= \left( \sum_{\bm{0}\neq \bsh \in \Z^d} \left[ \frac{\M_d(\bsh)!}{\# \S_d} \, r^{-1}_{\alpha,\bm{\beta}}(\bsh) \right]^{1/\lambda} \right)^{\lambda}.
	\end{align}
\end{prop}
\begin{remark}Some remarks are in order:
\begin{itemize}
 \item[(i)] In particular, for $\lambda=1$, the above proposition implies that there exists a shifted rank-$1$ lattice rule such that
\begin{equation}\label{bound_mse}
	E(Q_{d,n}(\bm{z^*}))
	\leq \sqrt{2c_R} \, C_{d,1}(r_{\alpha,\bm{\beta}})^{1/2} \, n^{-1/2}
\end{equation}
and it can be seen (cf.\ \cite[Corollary~4.14]{NSW14}) that this bound differs from \link{errorBound} only by a small factor which does not depend on~$d$, provided that \link{assump} is fulfilled.
\item[(ii)] If $1<\lambda<2\alpha$, then
(independent of $I_d$) there is an exponential dependence of $C_{d,\lambda}(r_{\alpha,\bm{\beta}})$ on the dimension~$d$ which gets stronger as $\lambda$ increases; this growth in the associated constant is typical and can be observed also when dealing with lattice rule constructions for classical spaces (without permutation-invariance).
For $\lambda\to 2\alpha$\tr{,} we even have $C_{d,\lambda}(r_{\alpha,\bm{\beta}})\to \infty$. Nevertheless, without going into details\tr{,} we mention that the dependence of this constant on the dimension can be controlled using additional constraints on the parameters $\bm{\beta}=(\beta_0,\beta_1)$. We stress that
these additional constraints are rather mild and would not achieve a similar tractability behavior for the full space.
For an extensive discussion of lower and upper bounds for $C_{d,\lambda}(r_{\alpha,\bm{\beta}})$ we refer to \cite[Proposition~4.12]{NSW14}.\hfill $\square$
\end{itemize} 
\end{remark}

\subsection{The component-by-component construction}\label{subsect:cbc}
Here we derive a component-by-component (CBC) construction to search for a generating vector $\bsz^*\in\Z_n^d$ such that (for some well-chosen shift $\bsDelta^*\in[0,1)^d$) 
the corresponding shifted rank-$1$ lattice rule $Q_{d,n}(\bsz^*)+\bsDelta^*$ satisfies an error bound similar to the one given in \autoref{prop:existence}. 
Our approach is motivated by similar constructions that exist for standard spaces defined via decay conditions of Fourier coefficients\tr{;} see, e.g., \cite{K03, SKJ02, SR02}.

We will \tr{further} make use of the following notation. 
Let $d\in\N$ and $I_d\subseteq\{1,\ldots,d\}$. For subsets $\emptyset \neq \setu \subseteq\{1,\ldots,d\}$ and vectors $\bsh=(h_1,\ldots,h_d)\in\Z^d$\tr{,} we denote by $\bsh_\setu:=(h_j)_{j\in\setu}$ the restriction of $\bsh$ to $\setu$. Hence, $\bsh_\setu\in\Z^\setu$ is a $\#\setu$-dimensional integer vector indexed by coordinates in $\setu$.
Given $\bsh_\setu\in\Z^\setu$, its trivial extension to the index set $\{1,\ldots,d\}$ is denoted by $(\bsh_\setu;\bm{0})\in\Z^d$\tr{;} i.e.,
\begin{equation*}
	(\bsh_\setu;\bm{0})_\setu := \bsh_\setu
	\qquad \text{and} \qquad
	(\bsh_\setu;\bm{0})_\ell := 0
	\quad \text{for all} \quad \ell\in\{1,\ldots,d\}\setminus\setu.
\end{equation*}
Finally, the restriction of the set of admissible permutations and their associated multiplicities (see \autoref{sect:setting} for the original definitions) to the subset~$\setu$ will be abbreviated by
\begin{equation*}
	\S_{\setu,I_d} 
	:= \S_{\setu}(I_d\cap \setu)
	:= \left\{ P\colon \setu \nach\setu \sep P \text{ a bijection such that } P\big|_{\setu\setminus (I_d\cap\setu)} = \id \right\}
\end{equation*}
and
\begin{equation*}
	\M_{\setu,I_d}(\bsh_\setu)! 
	:= \M_{\setu}(\bsh_\setu, I_d\cap\setu)!
	:= \#\left\{ P\in \S_{\setu,I_d} \sep P(\bsh_\setu)=\bsh_\setu \right\}, \qquad \bsh_\setu \in \Z^\setu,
\end{equation*}
respectively.

To derive the component-by-component construction, we will need a little preparation.
First, let us recall a technical assertion which can be found in \cite[Lemma~4.10]{NSW14}.
\begin{lemma}\label{char_prop}
	Let $d\in\N$, $\bsh \in \Z^d$, and $n\in\N$ \tr{be} prime. 
	Then
	\begin{align*}
		\frac{1}{\#\Z_n^d} \sum_{\bsz\in\Z_n^d} \ind{\bsh\in L(\bsz,n)^\bot}
		= \frac{1}{n} \sum_{j=0}^{n-1} \prod_{\ell=1}^{d} \frac{1}{n} \sum_{z_\ell=0}^{n-1} \exp(\twopii \, j h_\ell z_\ell / n) 
		= \begin{cases}
		    1 & \text{if } \bsh \equiv \bszero \pmod{n}, \\
		   \displaystyle
		    \frac{1}{n} & \text{otherwise},
		  \end{cases}
	\end{align*}
	where $L(\bsz,n)^\bot$ denotes the dual lattice induced by $\bsz$ and $\bsh \equiv \bszero \tpmod{n}$ is a shorthand for $h_\ell \equiv 0 \pmod{n}$ for all $\ell=1,\ldots,d$.
\end{lemma}
\begin{proof}
The proof is based on the character property given in (\ref{eq:char}). We refer to \cite{NSW14} for the complete proof. \qedhere
\end{proof}

Furthermore, we will use a dimension-wise decomposition of the mean squared worst case error \link{eq:mse}, as given below.
\begin{prop}\label{lem:decomp}
	Let $d\in\N$, $I_d\subseteq\{1,\ldots,d\}$, and $n\in\N$ be prime. 
	Then for all generating vectors $\bsz=(z_1,\ldots,z_d)\in \Z_n^d$ the mean squared worst case error of $Q_{d,n}(\bsz)+\bsDelta$ w.r.t.\ all shifts $\bsDelta\in[0,1)^d$ is given by
	\begin{equation}\label{eq:mse3}
		E(Q_{d,n}(\bsz))^2 
		:= \beta_0^d \, \sum_{\ell=1}^d B_{I_d,\ell}(z_1,\ldots,z_\ell),
	\end{equation}
	where we set
	\begin{equation}\label{def_B}
		B_{I_d,\ell}(z_1,\ldots,z_\ell)
		:= \sum_{\substack{\setu \subseteq\{1,\ldots,\ell\},\\\ell\in\setu}} c_{\setu,I_d}^{-1} \sum_{\bsh_{\setu}\in(\Z\setminus\{0\})^\setu} \frac{\M_{\setu,I_d}(\bsh_\setu)!}{\#\S_{\setu,I_d}} \, r_{\alpha,\bm{\beta}}^{-1}(\bsh_\setu) \, 
		\ind{\bsh_\setu \in L(\bsz_\setu,n)^\bot}
	\end{equation}
	for $\ell=1,\ldots,d$ and $c_{\setu,I_d} = \beta_0^{\#\setu} \binom{\# I_d}{\#(I_d\cap\setu)}$.
\end{prop}

\begin{proof}
Let $\emptyset \neq \setu \subseteq \{1,\ldots,d\}$ and assume that $\bsh=(h_1,\ldots,h_d)\in\Z^d\setminus\{\bm{0}\}$ with $h_j\neq 0$ if and only if $j\in\setu$ and $h_j=0$ otherwise.
Then, clearly,
\begin{equation*}
	r_{\alpha,\bm{\beta}}^{-1}(\bsh) 
	= \beta_0^{d-\#\setu} \, r_{\alpha,\bm{\beta}}^{-1}(\bsh_\setu)
	\qquad \text{and, for all } \bsz\in\Z_n^d, \qquad
	\bsh\cdot\bsz = \bsh_\setu \cdot \bsz_\setu\tr{;} 
\end{equation*}
i.e., $\bsh\in L(\bsz,n)^\bot$ if and only if $\bsh_\setu\in L(\bsz_\setu,n)^\bot$.
Moreover, for all such $\bsh=(\bsh_\setu;\bm{0})$ we have
\begin{align*}
	\frac{\M_d(\bsh)!}{\#\S_d} 
	&= \frac{\M_{ \{1,\ldots,d\} }((\bsh_\setu;\bm{0}), I_d)!}{\# I_d !} 
	= \frac{\#(I_d\setminus\setu)! \, \M_{ \setu, I_d }(\bsh_\setu)!}{\# I_d !} 
	= \binom{\#I_d}{\#(I_d\cap\setu)}^{-1} \, \frac{\M_{ \setu, I_d }(\bsh_\setu)!}{\#\S_{\setu,I_d}},
\end{align*}
since $\#\S_{\setu,I_d}=\# (I_d\cap\setu)!$.
Further, note that the collection of all non-empty $\setu\subseteq\{1,\ldots,d\}$ can be written as the disjoint union $\bigcup_{\ell=1}^d \left\{\setu \subseteq\{1,\ldots,\ell\} \sep \ell\in\setu\right\}$. This gives rise to the disjoint decomposition
\begin{equation*}
	\Z^d\setminus\{\bm{0}\} 
	= \bigcup_{\emptyset\neq \setu\subseteq\{1,\ldots,d\}} \left\{ (\bsh_\setu;\bm{0}) \sep \bsh_\setu \in (\Z\setminus\{0\})^\setu \right\}
	= \bigcup_{\ell=1}^d \bigcup_{\substack{ \setu\subseteq\{1,\ldots,\ell\},\\\ell\in\setu }} \left\{ (\bsh_\setu;\bm{0}) \sep \bsh_\setu \in (\Z\setminus\{0\})^\setu \right\}.
\end{equation*}

Next, observe that for every function $G\colon \Z^d\nach \C$ the following holds:
\begin{equation*}
	\sum_{\bsk \in \nabla_d} \frac{1}{\M_d(\bsk)!} \sum_{P \in \S_d} G(P(\bsk))
	= \sum_{\bsh \in \Z^d} G(\bsh).
\end{equation*}
Using this and the permutation-invariance of $\M_d(\cdot)!$ and $r_{\alpha,\bm{\beta}}(\cdot)$, we
can express the mean squared worst case error \link{eq:mse} as
\begin{align*}
	E(Q_{d,n}(\bsz))^2 
	&= -r^{-1}_{\alpha,\bm{\beta}}(\bm{0}) + \sum_{\bsk \in \nabla_d} \frac{1}{\M_d(\bsk)!} \sum_{P \in S_d} \M_d(P(\bsk))! \, \frac{r^{-1}_{\alpha,\bm{\beta}}(P(\bsk))}{\#\S_d} \, \ind{P(\bsk) \in L(\bsz,n)^{\perp}} \nonumber\\
	&= -r^{-1}_{\alpha,\bm{\beta}}(\bm{0}) + \sum_{\bsh \in \Z^d} \frac{r^{-1}_{\alpha,\bm{\beta}}(\bsh)}{\#\S_d} \, \M_d(\bsh)! \, \ind{\bsh \in L(\bsz,n)^{\perp}}.  
\end{align*}
From the above considerations we thus infer that
 \begin{align*}
 	E(Q_{d,n}(\bsz))^2 
	&= \sum_{\bsh\in\Z^d\setminus\{0\}} \frac{\M_d(\bsh)!}{\#\S_d} \, r_{\alpha,\bm{\beta}}^{-1}(\bsh) \, \ind{\bsh\in L(\bsz,n)^\bot}	\\
	&= \sum_{\ell=1}^d \sum_{\substack{ \setu\subseteq\{1,\ldots,\ell\},\\\ell\in\setu }} \sum_{\bsh_{\setu}\in(\Z\setminus\{0\})^\setu} \frac{\M_d((\bsh_\setu;\bm{0}))!}{\#\S_d} \, r_{\alpha,\bm{\beta}}^{-1}((\bsh_\setu;\bm{0})) \, \ind{(\bsh_\setu;\bm{0})\in L(\bsz,n)^\bot} \\
	&= \sum_{\ell=1}^d \sum_{\substack{ \setu\subseteq\{1,\ldots,\ell\},\\\ell\in\setu }} \! \beta_0^{d-\#\setu}\binom{\#I_d}{\#(I_d\cap\setu)}^{-1} \!\!\ns \ns\sum_{\bsh_{\setu}\in(\Z\setminus\{0\})^\setu} \! \frac{\M_{\setu,I_d}(\bsh_\setu)!}{\#\S_{\setu,I_d}} \, r_{\alpha,\bm{\beta}}^{-1}(\bsh_\setu) \, \ind{\bsh_\setu\in L(\bsz_\setu,n)^\bot},
\end{align*}
which proves the \tr{result}.
\end{proof}

Now we are well-prepared to state and prove the main theorem of this section. 
It \tr{presents} a component-by-component construction to search for a generating vector $\bsz^*\in\Z_n^d$ such that (for some well-chosen shift $\bsDelta^*\in[0,1)^d$) the error of the shifted 
rank-$1$ lattice rule $Q_{d,n}(\bsz^*)+\bsDelta^*$ achieves a rate of convergence which is arbitrarily close to $\0(n^{-\alpha})$. 

\begin{theorem}[CBC construction]\label{thm:cbc}
	Let $d\in\N$, $I_d\subseteq\{1,\ldots,d\}$, and assume $n\in\N$ with $n\geq c_R$ to be prime. 
	Moreover, let $z_1^*\in\{1,\ldots,n-1\}$ be arbitrarily fixed and select $z_2^*,\ldots,z_d^*\in\Z_n$ component-by-component via minimizing the quantities $B_{I_d,\ell}(z_1^*,\ldots,z_{\ell-1}^*,z_\ell)$, defined in \eqref{def_B}, subject to $z_\ell\in\Z_n$, $\ell=2,\ldots,d$.
	Then there exists $\bsDelta^*\in[0,1)^d$ such that the shifted rank-$1$ lattice rule $Q_{d,n}=Q_{d,n}(z_1^*,\ldots,z_d^*)+\bsDelta^*$ for integration of $I_d$-permutation-invariant functions in $F_d(r_{\alpha,\bm{\beta}})$ satisfies   
	\begin{equation}\label{opt_bound}
		e^{\mathrm{wor}}(Q_{d,n}; \SI_{I_d}(F_d(r_{\alpha,\bm{\beta}})))^2
		\,\leq\, (1+c_R)^{\lambda}\,C_{d,\lambda}(r_{\alpha,\bm{\beta}})  \max\{1,\#I_d\} \frac{1}{n^{\lambda}}
	\end{equation}
	for all $1 \leq \lambda < 2\alpha$,	where $C_{d,\lambda}(r_{\alpha,\bm{\beta}})$ is given by \link{C_dlambda}.
\end{theorem}

Before presenting the proof of \autoref{thm:cbc}, we stress that our bound \link{opt_bound} for the CBC construction  is only slightly worse than the error bound found in the general existence result \tr{of} \autoref{prop:existence}.
It always \tr{depends} on the number of permutation invariant variables and thus cannot be used to deduce strong polynomial tractability. 
This seems to be unavoidable, see also \autoref{rem:CBC} below. 
However, note that this additional linear dependence on $d$ is a noticeable overhead only for the case of $\lambda=1$; 
for $\lambda>1$, the exponential growth in $C_{d,\lambda}(r_{\alpha,\bm{\beta}})$ overshadows this dependence on $d$.

\begin{proof}[Proof of \autoref{thm:cbc}]
Our proof is based on the dimension-wise decomposition of the mean squared worst case error given in \autoref{lem:decomp}. 
Once we have found $\bsz^*=(z_1^*,\ldots,z_d^*)\in\Z_n^d$ such that the corresponding quantity $E(Q_{d,n}(\bsz^*))^2$ is upper bounded by the right-hand side of \link{opt_bound},
 the result follows (as usual) by the mean value theorem which ensures the existence of a shift $\bm{\Delta}^*$ with
\begin{equation}\label{better_than_avg}
	e^{\mathrm{wor}}(Q_{d,n}(\bsz^*)+\bsDelta^*; \SI_{I_d}(F_d(r_{\alpha,\bm{\beta}})))^2 
	\leq E(Q_{d,n}(\bsz^*))^2.
\end{equation}

\emph{Step 1.} 
Let $\lambda \in [1,2\alpha)$, $d\in\N$, $I_d\subseteq\{1,\ldots,d\}$, and $n$ be a fixed \tr{prime}.
We apply Jensen's inequality (see \autoref{lem:jensens}) with $p=1\geq 1/\lambda=q$ to the expression in
~\link{eq:mse3} and obtain
\begin{equation*}
	E(Q_{d,n}(\bsz))^{2/\lambda} 
	\leq \beta_0^{d/\lambda} \, \sum_{\ell=1}^d B_{I_d,\ell}(z_1,\ldots,z_\ell)^{1/\lambda} \quad \text{for all }\quad \bsz\in \Z_n^d,
\end{equation*}
 with $B_{I_d,\ell}$ as defined in \link{def_B}.
In Step 2 and 3 below we will show that if we select~$\bsz=\bsz^*$ component-by-component, then for all $\ell=1,\ldots,d$\tr{,} the summands in the estimate are bounded by
\begin{align}
	&B_{I_d,\ell}(z_1^*,\ldots, z_\ell^*)^{1/\lambda}\label{bound_B_lambda}\\
	&\qquad\qquad \leq (1+c_R) \,\max\{1,\#I_d\}^{1/\lambda} \frac{1}{n} \sum_{\substack{\setu \subseteq\{1,\ldots,\ell\},\\\ell\in\setu}} c_{\setu,I_d}^{-1/\lambda} 
\sum_{\bsh_{\setu}\in(\Z\setminus\{0\})^\setu} \left(\frac{\M_{\setu,I_d}(\bsh_\setu)!}{\#\S_{\setu,I_d}}\right)^{1/\lambda} \, r_{\alpha,\bm{\beta}}^{-1/\lambda}(\bsh_\setu). \nonumber
\end{align}
The recombination of these building blocks as in 
\autoref{lem:decomp} then yields
\begin{align*}
	E(Q_{d,n}(\bsz^*))^{2/\lambda} 
	&\leq (1+c_R)  \left( \sum_{\bsh\in\Z^d\setminus\{0\}} \left[\frac{\M_d(\bsh)!}{\#\S_d} \, r_{\alpha,\bm{\beta}}^{-1}(\bsh) \right]^{1/\lambda} \right)  \max\{1,\#I_d\}^{1/\lambda}  \,\frac{1}{n}, \label{bound_E2l}
\end{align*}
such that the claim \link{opt_bound} is implied by \link{better_than_avg}.

\emph{Step 2.} 
To prove \link{bound_B_lambda} we apply Jensen's inequality (again for $p=1\geq 1/\lambda=q$) to \link{def_B} and obtain
\begin{equation}\label{est:JensenB}
		B_{I_d,\ell}(z_1,\ldots,z_\ell)^{1/\lambda}
		\leq \sum_{\substack{\setu \subseteq\{1,\ldots,\ell\},\\\ell\in\setu}} c_{\setu,I_d}^{-1/\lambda} \sum_{\bsh_{\setu}\in(\Z\setminus\{0\})^\setu} \left( \frac{\M_{\setu,I_d}(\bsh_\setu)!}{\#\S_{\setu,I_d}} \right)^{1/\lambda} \, r_{\alpha,\bm{\beta}}^{-1/\lambda}(\bsh_\setu) \, 
		\ind{\bsh_\setu \in L(\bsz_\setu,n)^\bot}
\end{equation}
for all $\bsz\in\Z_n^d$ and $\ell=1,\ldots,d$.

Now consider $\ell=1$ and assume $z_1=z_1^*\in\{1,2,\ldots,n-1\}$ to be fixed arbitrarily. Then the sum over all subsets in the latter bound reduces to $\setu=\{1\}$. 
Accordingly, we have $\S_{\setu,I_d}=\{\id\}$ and 
$\M_{\setu,I_d}(\bsh_\setu)!=1$ does not depend on~$\bsh_\setu$. 
Moreover, for $\ell=1$ only those non-trivial indices $\bsh_\setu$ belong to $L(\bsz_\setu,n)^\bot$, i.e., satisfy $\bsh_\setu \cdot \bsz_\setu \equiv 0 \pmod{n}$, which can be written as $\bsh_\setu=h_1=n\,k$ for some $k\in\Z\setminus\{0\}$ because $\bsz_\setu=z_1^*\in\Z_n$ and $n$ is assumed to be prime. 
Hence, in this case 
we obtain that $r_{\alpha,\bm{\beta}}^{-1/\lambda}(\bsh_\setu)$ equals
\begin{align}
	r_{\alpha,\bm{\beta}}^{-1/\lambda}(n\, k) 
	= \beta_1^{1/\lambda} R(n\abs{k})^{-2\alpha/\lambda}
	\leq \beta_1^{1/\lambda} \left[ \frac{n}{c_R} \, R(\abs{k}) \right]^{-2\alpha/\lambda}
	= \left[\frac{c_R}{n}\right]^{2\alpha/\lambda} r_{\alpha,\bm{\beta}}^{-1/\lambda}(k)
	\leq \frac{c_R}{n} \; r_{\alpha,\bm{\beta}}^{-1/\lambda}(k) \label{est:r}
\end{align}
with $c_R/n \leq (1+c_R)\, \max\{1,\#I_d\}^{1/\lambda} / n$ (recall that $\lambda < 2\alpha$ and $n\geq c_R$).
Consequently, we have
\begin{align*}
	B_{I_d,1}(z_1^*)^{1/\lambda}
	&\leq c_{\{1\},I_d}^{-1/\lambda} \sum_{\substack{h_1=n\, k,\\ k \in \Z\setminus\{0\}}} r_{\alpha,\bm{\beta}}^{-1/\lambda}(h_1)\\
	&\leq (1+c_R) \,\max\{1,\#I_d\}^{1/\lambda} \frac{1}{n} \sum_{\substack{\setu \subseteq\{1\},\\1\in\setu}} c_{\setu,I_d}^{-1/\lambda} \sum_{\bsh_{\setu}\in(\Z\setminus\{0\})^\setu} \left( \frac{\M_{\setu,I_d}(\bsh_\setu)!}{\#\S_{\setu,I_d}} \right)^{1/\lambda} r_{\alpha,\bm{\beta}}^{-1/\lambda}(\bsh_\setu),
\end{align*}
where $k$ was relabeled as $\bsh_\setu$. 
In other words, \link{bound_B_lambda} holds true for $\ell=1$.

\emph{Step 3.} 
\tr{Now}, assume that we have already determined $z_1^*,\ldots,z_{\ell-1}^*$ for some $\ell\in\{2,\ldots,d\}$.
Then the best choice $z_\ell^* \in \Z_n$, i.e., the minimizer of $B_{I_d,\ell}(z_1^*,\ldots,z_{\ell-1}^*, \cdot\,)$, satisfies
\begin{align*}
	&B_{I_d,\ell}(z_1^*,\ldots,z_{\ell-1}^*, z_\ell^*)^{1/\lambda} \\
	&\quad \leq \frac{1}{\# \Z_n} \sum_{z_\ell \in \Z_n} B_{I_d,\ell}(z_1^*,\ldots,z_{\ell-1}^*, z_\ell)^{1/\lambda} \\
	&\quad \leq \sum_{\substack{\setu \subseteq\{1,\ldots,\ell\},\\\ell\in\setu}} c_{\setu,I_d}^{-1/\lambda} \sum_{\bsh_{\setu}\in(\Z\setminus\{0\})^\setu} \left( \frac{\M_{\setu,I_d}(\bsh_\setu)!}{\#\S_{\setu,I_d}} \right)^{1/\lambda} r_{\alpha,\bm{\beta}}^{-1/\lambda}(\bsh_\setu) \left[ \frac{1}{\# \Z_n} \sum_{z_\ell \in \Z_n} \ind{\bsh_\setv\cdot\bsz_\setv^*+h_\ell z_\ell\equiv 0\tpmod{n}} \right]\!,
\end{align*}
where we used \link{est:JensenB} and employed the notation $\setv:=\setu\setminus\{\ell\}$ and $\bsz^*_\setv := (z_j^*)_{j\in\setv}$, as well as the shorthands $\bsh_\setv:=(\bsh_\setu; \bm{0})_\setv$ and $h_\ell:=(\bsh_\setu)_\ell$.

Next, we estimate the term in the brackets for every $\bsh_{\setu}\in(\Z\setminus\{0\})^\setu$ with $\setu\subseteq\{1,\ldots,\ell\}$ which contains $\ell$. 
The character property~\link{eq:char} yields
\begin{align*}
	\frac{1}{\# \Z_n} \sum_{z_\ell \in \Z_n} \ind{\bsh_\setv\cdot\bsz_\setv^*+h_\ell z_\ell\equiv 0\tpmod{n}}
	&= \frac{1}{n} \sum_{z_\ell=0}^{n-1} \frac{1}{n} \sum_{j=0}^{n-1} \exp(\twopii \, j (\bsh_\setv \cdot z_\setv^* + h_\ell z_\ell) / n) \\
	&= \frac{1}{n} \sum_{j=0}^{n-1} \exp(\twopii \, j (\bsh_\setv \cdot z_\setv^*) / n) \, \frac{1}{n} \sum_{z_\ell=0}^{n-1} \exp(\twopii \, j h_\ell z_\ell / n) \\
	&\leq \frac{1}{n} \sum_{j=0\,}^{n-1} \underbrace{ \, \Bigg|\exp(\twopii \, j (\bsh_\setv \cdot z_\setv^*) / n) \Bigg| \, }_{\quad \le 1} \Bigg| \underbrace{ \frac{1}{n} \sum_{z_\ell=0}^{n-1} \exp(\twopii \, j h_\ell z_\ell / n) }_{\quad = \ind{jh_\ell\equiv 0\tpmod{n}} \geq 0} \Bigg| \\	
	&\leq W_n(h_\ell),
\end{align*}
where
\begin{equation*}
	W_n(h_\ell)
	:= \frac{1}{n} \sum_{j=0}^{n-1} \frac{1}{n} \sum_{z_\ell=0}^{n-1} \exp(\twopii \, j h_\ell z_\ell / n)
	= \begin{cases}
		1 & \text{ if } h_\ell\equiv 0\pmod{n},\\
		1/n & \text{ otherwise},
	\end{cases}
\end{equation*}
due to \autoref{char_prop} for $d=1$.

To exploit this estimate, given $\setu=\setv\cup\{\ell\}$ as above, we split up the sum over all $\bsh_\setu=(\bsh_\setv, h_\ell)$ from $(\Z\setminus\{0\})^\setu=(\Z\setminus\{0\})^\setv \times (\Z\setminus\{0\})$ and obtain
\begin{align*}
	&\sum_{\bsh_{\setu}\in(\Z\setminus\{0\})^\setu} \left(\frac{\M_{\setu,I_d}(\bsh_\setu)!}{\#\S_{\setu,I_d}}\right)^{1/\lambda} \, r_{\alpha,\bm{\beta}}^{-1/\lambda}(\bsh_\setu) \, W_n(h_\ell) \\
	&\qquad\qquad = \sum_{\bsh_{\setv}\in(\Z\setminus\{0\})^\setv} \sum_{\substack{h_\ell\in(\Z\setminus\{0\})\\h_\ell\equiv 0 \tpmod{n}}} \left(\frac{\M_{\setu,I_d}((\bsh_\setv, h_\ell))!}{\#\S_{\setu,I_d}}\right)^{1/\lambda} \, r_{\alpha,\bm{\beta}}^{-1/\lambda}((\bsh_\setv, h_\ell)) \\
	&\qquad\qquad\qquad+ \frac{1}{n} \sum_{\substack{\bsh_{\setu}\in(\Z\setminus\{0\})^\setu\\ h_\ell\not\equiv 0 \tpmod{n}}} \left(\frac{\M_{\setu,I_d}(\bsh_\setu)!}{\#\S_{\setu,I_d}}\right)^{1/\lambda} \ r_{\alpha,\bm{\beta}}^{-1/\lambda}(\bsh_\setu).
\end{align*}
In the first term every $h_\ell$ equals $n\,k$ for some $0\neq k \in\Z$. 
Thus, we can perform a change of variables similar to what was done in Step 2 and replace the inner summation by a sum over all $k\in \Z\setminus\{0\}$. 
Then from the product structure of $r_{\alpha,\bm{\beta}}(\cdot)$ and the bound \link{est:r} it follows that
\begin{equation*}
	r_{\alpha,\bm{\beta}}^{-1/\lambda}((\bsh_\setv, h_\ell)) 
	= r_{\alpha,\bm{\beta}}^{-1/\lambda}(\bsh_\setv)  \, r_{\alpha,\bm{\beta}}^{-1/\lambda}(n \, k)
	\leq \frac{c_R}{n} \, r_{\alpha,\bm{\beta}}^{-1/\lambda}(\bsh_\setv)  \, r_{\alpha,\bm{\beta}}^{-1/\lambda}(k) 
	= \frac{c_R}{n} \, r_{\alpha,\bm{\beta}}^{-1/\lambda}((\bsh_\setv, k)).
\end{equation*}
But now we also need to estimate $\M_{\setu,I_d}((\bsh_\setv, h_\ell))!$ w.r.t.\ this transformation of variables. 
If $\ell \notin I_d$, then replacing $h_\ell=n\, k$ by $k$ does not effect $\M_{\setu,I_d}$. On the other hand, if $\ell\in I_d$, then, at least 
we know that 
\begin{align*}
	\M_{\setu,I_d}((\bsh_\setv, n\,k))!
	\leq \max\{1,\#(I_d\cap\setu)\} \cdot \M_{\setu,I_d}((\bsh_\setv, k))! 
	\leq \max\{1,\#I_d\} \cdot \M_{\setu,I_d}((\bsh_\setv, k))!.
\end{align*}

Hence, dropping the condition $h_\ell \not\equiv 0 \pmod{n}$ in the second term finally yields
\begin{align*}
	&\sum_{\bsh_{\setu}\in(\Z\setminus\{0\})^\setu} \left(\frac{\M_{\setu,I_d}(\bsh_\setu)!}{\#\S_{\setu,I_d}}\right)^{1/\lambda} \, r_{\alpha,\bm{\beta}}^{-1/\lambda}(\bsh_\setu) \, W_n(h_\ell) \\
	&\qquad\qquad \leq \max\{1,\#I_d\}^{1/\lambda} \, \frac{c_R}{n} \sum_{\bsh_{\setv}\in(\Z\setminus\{0\})^\setv} \sum_{k\in(\Z\setminus\{0\})} \left(\frac{\M_{\setu,I_d}((\bsh_\setv, k))!}{\#\S_{\setu,I_d}}\right)^{1/\lambda} \, r_{\alpha,\bm{\beta}}^{-1/\lambda}((\bsh_\setv, k)) \\
	&\qquad\qquad\qquad+ \frac{1}{n} \sum_{\bsh_{\setu}\in(\Z\setminus\{0\})^\setu} \left(\frac{\M_{\setu,I_d}(\bsh_\setu)!}{\#\S_{\setu,I_d}}\right)^{1/\lambda} \ r_{\alpha,\bm{\beta}}^{-1/\lambda}(\bsh_\setu) \\
	&\qquad\qquad \leq (1+c_R)\,\max\{1,\#I_d\}^{1/\lambda} \, \frac{1}{n} \sum_{\bsh_{\setu}\in(\Z\setminus\{0\})^\setu} \left(\frac{\M_{\setu,I_d}(\bsh_\setu)!}{\#\S_{\setu,I_d}}\right)^{1/\lambda} \ r_{\alpha,\bm{\beta}}^{-1/\lambda}(\bsh_\setu)
\end{align*}
for every subset $\setu\subseteq\{1,\ldots,\ell\}$ which contains $\ell$.
This immediately implies the desired bound~\link{bound_B_lambda} on $B_{I_d,\ell}(z_1^*,\ldots, z_\ell^*)^{1/\lambda}$ and the proof is complete.
\end{proof}

\begin{remark}\label{rem:CBC}
Let us add some final remarks on the CBC construction:
\begin{itemize}
	\item[(i)] Note that this theorem asserts the existence of a $\bsz^*$ that achieves the desired bound and gives a constructive way of finding it (by successively minimizing the quantities $B_{I_d,\ell}(z_1,\ldots,z_\ell)$ subject to $z_\ell$).
	\item[(ii)] As usual, all non-trivial choices for the first component $z_1^*$ of the generating vector $\bsz^*\in\Z_n^d$ are equivalent.
	\item[(iii)] For the remaining components we actually assumed more than we needed: Step 3 in the above proof shows that instead of minimizing $B_{I_d,\ell}(z_1^*,\ldots,z_{\ell-1}^*,z_\ell)$ for $z_\ell\in\Z_n$, $\ell\in\{2,\ldots,d\}$, it would be sufficient to select $z_\ell^*$ which performs better than the average (w.r.t.\ $\lambda$), i.e., which fulfills 
	\begin{equation}\label{avg_B}
		B_{I_d,\ell}(z_1^*,\ldots,z_{\ell-1}^*, z_\ell^*)^{1/\lambda} 
		 \leq \frac{1}{\# \Z_n} \sum_{z_\ell \in \Z_n} B_{I_d,\ell}(z_1^*,\ldots,z_{\ell-1}^*, z_\ell)^{1/\lambda}.
	\end{equation}
	Moreover note that from \tr{H\"older's} inequality it easily follows that the latter estimate yields a corresponding bound for all $\widetilde{\lambda} \leq \lambda$. 
	That is, selecting $z^*_\ell$, $\ell=2,\ldots,d$, according to \link{avg_B} implies \link{opt_bound} with $\lambda$ replaced by any $\widetilde{\lambda}\in [1,\lambda]$, whereas minimizing $B_{I_d,\ell}$ produces a lattice rule which satisfies the bound \link{opt_bound} for the whole range $[1,2\alpha)$ of $\lambda$.
	\item[(iv)] Observe that (in general) a lattice rule based on a generating vector $\bsz^*$ constructed as above is unfortunately \emph{not} extensible in the dimension. 
The reason is that with increasing~$d$ the factors $c_{\setu,I_d}$ in the definition of $B_{I_d,\ell}$ (see \link{def_B}) will change such that the contributions of the subsets $\setu\subseteq\{1,\ldots,\ell\}$ with $\ell\in\setu$ are different for every dimension. Hence, for fixed $\ell$ the optimal choice $z_{\ell}^*$ may change with varying $d$ unless the set of permutation-invariant coordinates $I_d$ is independent of the dimension (which is the uninteresting case).
	Thus, to successfully apply the CBC construction, we need to know the target dimension in advance.
	Nevertheless, compared to a full search in $\Z_n^d$ the complexity of finding a suitable generating vector is significantly reduced if we do it component-by-component and parts of the calculations made for dimension $d$ may be reused while processing the $(d+1)$'th dimension.
	\item[(v)] As mentioned earlier, our conjecture is that the dependence of \link{opt_bound} on the dimension~$d$ is inevitable and cannot be improved by alternate steps in our proof. 
The CBC construction tries to distinguish between otherwise identical dimensions; so no CBC construction of lattice rules for the integration of permutation-invariant functions can satisfy bounds like \link{bound_mse}, in our opinion. 
Indeed, \link{def_B} in \autoref{lem:decomp} above already indicates the complicated influence of already selected components $z_1^*,\ldots,z_{\ell-1}^*$ on the error contribution induced by the choice of the current coordinate $z_\ell$ (and forthcoming indicee). 
The fact that the problem is caused by the permutation-invariance assumption ($\M_{\setu,I_d}$ is not of product structure!) is nicely reflected by the factor $\max\{1,\#I_d\}$ in \link{opt_bound} which disappears for standard constructions, i.e., for spaces without symmetry constraints.
	\item[(vi)] Finally, let us mention that slightly stronger assumptions allow to get rid of the maximum term in \link{opt_bound}. 
Indeed, a careful inspection of the proof shows that actually \link{opt_bound} can be improved to
	\begin{align*}
		e^{\mathrm{wor}}(Q_{d,n}; \SI_{I_d}(F_d(r_{\alpha,\bm{\beta}})))^2
		&\leq \left[ \max\{1,\#I_d\}^{1/\lambda} \left(\frac{c_R}{n}\right)^{2\alpha/\lambda} + \frac{1}{n} \right]^{\lambda} \, C_{d,\lambda}(r_{\alpha,\bm{\beta}}) \\
		&= \left[ \frac{\max\{1,\#I_d\}^{1/\lambda} \, c_R^{2\alpha/\lambda}}{n^{2\alpha/\lambda-1}} + 1 \right]^{\lambda} \, C_{d,\lambda}(r_{\alpha,\bm{\beta}}) \, \frac{1}{n^\lambda}.
	\end{align*}
	Thus, if we assume that $n \geq c_R \,\max\{1,\#I_d\}^{1/(2\alpha-\lambda)}$, then we exactly recover the error bound \link{est:existence} from \autoref{prop:existence}.
\hfill$\square$
\end{itemize}
\end{remark}

\section{An \tr{alternative} construction} 
\label{sect:alternate_construction}
In this section we consider a completely different approach towards efficient cubature rules for the integration problem defined in \autoref{sect:setting}.
For this purpose, in \autoref{subsect:approx} we collect some basic facts (taken from the monographs of Novak and Wo{\'z}niakowski \cite{NW08,NW10,NW12}) about $L_2$-\emph{approximation problems} on more general 
reproducing kernel Hilbert spaces (RKHSs) $H_d=H(K)$ of $d$-variate functions and on somewhat larger Banach spaces $B_d \supset H_d$. 
Then, in \autoref{subsect:quad}\tr{,  we show} that some knowledge about the quantities related to these approximation problems allows \tr{us} to (semi-explicitly) construct a 
sequence of cubature rules $Q_{d,N}$ for integration on $H_d$ with worst case errors that decay with the desired rate.
Finally, this construction is applied to our original integration problem of permutation-invariant functions in \autoref{subsect:korobov}.

\subsection{The $L_2$-approximation problem for RKHSs}\label{subsect:approx}
Let $K$ denote the reproducing kernel of some Hilbert space $H_d=H(K)$ of real (or complex) valued functions $f$ on $[0,1]^d$. The inner product on this space will be denoted by $\distr{\cdot}{\cdot}$ and $\norm{\cdot\sep H_d}=\sqrt{\distr{\cdot}{\cdot}}$ is the induced norm.
Then from the reproducing kernel property it follows that
\begin{equation*}
	\abs{f(\bsx)}=\abs{\distr{f}{K(\cdot,\bsx)}} \leq \norm{f\sep H_d} \, \sqrt{K(\bsx,\bsx)}
	\qquad \text{for all} \qquad \bsx\in[0,1]^d,
\end{equation*}
see, e.g., Aronszajn~\cite{A50} where a comprehensive discussion of RKHSs can be found.
Hence, if
\begin{equation}\label{def:M2}
	M_{2,d}\tr{:=}M_{2,d}(K)\tr{:=}\int_{[0,1]^d} K(\bsx,\bsx) \, \rd\bsx
\end{equation}
is finite, then the space $H_d$ \tr{can be} continuously embedded into the space $L_2([0,1]^d)$ and it holds
$\norm{f \sep L_2}^2 \leq M_{2,d} \norm{f \sep H_d}^2$ for every $f\in H_d$.
In other words, $M_{2,d}(K)^{1/2}$ serves as an upper bound for the operator norm
\begin{gather*}
		\norm{\mathrm{App}_d} \tr{:=} \sup_{f\in H_d, \norm{f\sep H_d}\leq 1} \norm{f \sep L_2}
		\qquad \text{of} \qquad
		\mathrm{App}_d \colon H_d \nach L_2([0,1]^d), \quad f\mapsto\mathrm{App}_df:=f.
\end{gather*}
If, in addition, the embedding $H_d\hookrightarrow L_2([0,1]^d)$ is compact, then we can try to approximate $\mathrm{App}_d$ by non-adaptive algorithms
\begin{gather}\label{linalgo}
		A_{d,n} f := \sum_{i=1}^n L_i(f) \,a_i
\end{gather}
which are linear combinations\footnote{Note that non-linear and/or 
adaptive algorithms don't give an advantage in the setting we are going to describe \tr{(see \cite[Chapter 5]{NW08} for details)}. Hence, our choice is w.l.o.g.} 
of at most $n$ information operations \tr{$L_1,\ldots,L_n$} from some class $\Lambda$
and arbitrary functions \tr{$a_1,\ldots,a_n\in L_2([0,1]^d)$}.
In the worst case setting, the error of such an algorithm is defined by
\begin{gather*}
		e^{\mathrm{wor}}(A_{d,n}; \mathrm{App}_d) 
		\tr{:=} \sup_{f\in H_d, \norm{f\sep H_d}\leq 1} \norm{\mathrm{App}_d f - A_{d,n}f\sep L_2}.
\end{gather*}
As long as we restrict ourselves to the class 
$\Lambda=\Lambda^{\mathrm{all}}=H_d^*$ of continuous linear functionals
the complexity of this problem is completely determined by the 
spectrum of the self-adjoint and positive semi-definite operator
$W_d := \mathrm{App}_d^{\dagger}\mathrm{App}_d$.
Due to the compactness of $\mathrm{App}_d$, $W_d$ is also compact such that
it possesses a countable set of eigenpairs $\{(\lambda_{d,m},\eta_{d,m}) \sep m\in\N \}$,
\begin{gather*}
		W_d \eta_{d,m} = \lambda_{d,m} \eta_{d,m}.
\end{gather*}
The eigenvalues $\lambda_{d,m}$ are non-negative real numbers which form 
a null-sequence that is (w.l.o.g.) ordered in a non-increasing way.
For the ease of presentation we restrict ourselves 
to the case where $H_d$ is separable,  $\dim H_d = \infty$, and $\lambda_{d,m}>0$ for all $m\in\N$.
Then the \tr{set} of eigenfunctions \tr{$\{\eta_{d,m}\sep m\in\N\}$ forms} an orthonormal basis of $H_d$.
Furthermore, they are also orthogonal w.r.t.\ the inner product of $L_2([0,1]^d)$ and
we have $\lambda_{d,m} = \norm{\eta_{d,m} \sep L_2}^2$ for any $m\in\N$.
The optimal algorithm $A_{d,n}^*$ in the mentioned (worst case $L_2$-approximation) setting is given by the choice
\begin{gather}\label{optalgo}
		L_i 
		= \distr{\,\cdot\,}{\eta_{d,i}}
		= \frac{\distr{\,\cdot\,}{\eta_{d,i}}_{L_2}}{\lambda_{d,i}} 
		\quad \text{and} \quad 
		a_i = \eta_{d,i}, \quad i=1,\ldots,n.
\end{gather}
Finally, its worst case error equals the $n$th minimal worst case error which can be calculated exactly:
\begin{gather*}
		e_{\mathrm{wor}}(n;\mathrm{App}_d, \Lambda^{\mathrm{all}}) 
		= e^{\mathrm{wor}}(A_{d,n}^*; \mathrm{App}_d) 
		= \sqrt{\lambda_{d,n+1}}, 
		\qquad n\in\N_0.
\end{gather*}
We refer to \cite{NW08} for a more detailed discussion.

All the statements made so far are true for arbitrary Hilbert spaces $H_d$ which are compactly embedded into $L_2([0,1]^d)$.
Taking into account the reproducing kernel property and using the definition of 
the adjoint operator $\mathrm{App}_d^{\dagger}$, it can be checked easily that $W_d$ takes the form
of an integral operator,
\begin{gather*}
		W_d f = \int_{[0,1]^d} f(\bsy) \, K(\cdot, \bsy) \, \rd\bsy, \qquad f \in H_d,
\end{gather*}
since for any $f \in H_d$ and all $\bsx\in[0,1]^d$
\begin{align*}
		(\mathrm{App}_d^{\dagger}\mathrm{App}_d f)(\bsx) 
		&= \distr{\mathrm{App}_d^{\dagger}\mathrm{App}_d f}{K(\cdot,\bsx)}_{H_d}
		= \distr{\mathrm{App}_d f}{\mathrm{App}_d K(\cdot,\bsx)}_{L_2}
		= \distr{f}{K(\cdot,\bsx)}_{L_2} \\
		&= (W_d f)(\bsx).
\end{align*}
Among other useful properties, we 
also have that if $M_{2,d}$ as defined in \link{def:M2} is finite, then
\begin{gather*}
		M_{2,d} 
		= \sum_{m\in\N} \norm{\eta_{d,m} \sep L_2}^2 
		= \sum_{m\in\N} \lambda_{d,m} 
		=: \mathrm{trace}(W_d),
\end{gather*}
see, e.g., \cite[Section~10.8]{NW10}.
In particular, this finite trace property immediately implies that $\lambda_{d,m} \in \0(m^{-1})$ 
or $e_{\mathrm{wor}}(n;\mathrm{App}_d, \Lambda^{\mathrm{all}}) \in \0(n^{-1/2})$, respectively.

In addition,  $M_{2,d}$ is also related to the 
\emph{average case approximation} setting.
To this end, assume that $B_d$ \tr{is} some separable Banach space of real-valued functions on $[0,1]^d$
equipped with a zero-mean Gaussian measure $\mu_d$ such that
its correlation operator $C_{\mu_d}\colon B_d^*\nach B_d$ applied to point evaluation functionals $\mathrm{L}_{\bsx}$
can be expressed in terms of
the reproducing kernel $K$ of $H_d$.
That is, we choose $B_d$ and $\mu_d$ such that
\begin{gather*}
		K(\bsx,\bsy) 
		= \mathrm{L}_{\bsx}(C_{\mu_d} \mathrm{L}_{\bsy}) 
		= \int_{B_d} f(\bsx) \, f(\bsy) \, \rd\mu_d(f)
		\quad \text{for all} \quad \bsx,\bsy \in [0,1]^d,
\end{gather*}
where $\mathrm{L}_{\bsx} \colon B_d \nach \R$ with $f\mapsto\mathrm{L}_{\bsx}(f):=f(\bsx)$ for $\bsx\in[0,1]^d$.
We stress 
that these assumptions imply a continuous embedding of $H_d$ into $B_d$.
For more details, the reader is referred to \cite[Appendix B]{NW08} and \cite[Section 13.2]{NW10}.

Again, we look for good approximations $A_{d,n} f$ to $f$ in the norm of $L_2([0,1]^d)$.
Thus, we formally approximate the operator
\begin{gather*}
		\widetilde{\mathrm{App}}_d \colon B_d \nach L_2([0,1]^d), \qquad f \mapsto \widetilde{\mathrm{App}}_df := f,
\end{gather*}
by algorithms of the form \link{linalgo}. 
Clearly, we need to make sure that $B_d$ is continuously embedded into $L_2([0,1]^d)$.
The difference to the worst case setting discussed above is that
this time the error will be measured by
some expectation with respect to $\mu_d$.
For this purpose, we define
\begin{gather*}
		e^{\mathrm{avg}}(A_{d,n}; \widetilde{\mathrm{App}}_d) 
		\tr{:=} \left( \int_{B_d} \norm{\widetilde{\mathrm{App}}_d f - A_{d,n}f\sep L_2}^2 \, \rd\mu_d(f) \right)^{1/2},
\end{gather*}
which immediately implies that the initial error is given by
\begin{align}
		e_{\mathrm{avg}}(0; \widetilde{\mathrm{App}}_d,\Lambda) 
		&= e^{\mathrm{avg}}(A_{d,0}; \widetilde{\mathrm{App}}_d) \label{def:e_avg}\\
		&= \left( \int_{B_d} \int_{[0,1]^d} f(\bsx)^2 \, \rd\bsx \rd\mu_d(f) \right)^{1/2}
		= \left( \int_{[0,1]^d} K(\bsx,\bsx) \, \rd\bsx \right)^{1/2}
		= M_{2,d}^{1/2}. \nonumber
\end{align}
Since $A_{d,0}\equiv 0$ does not use any information on $f$, this holds for information from both classes $\Lambda$ in $\{\Lambda^{\mathrm{all}},\Lambda^{\mathrm{std}}\}$, 
where $\Lambda^{\mathrm{std}}$ denotes the set of all function evaluation functionals and (as before) $\Lambda^{\mathrm{all}}$ is the set of all continuous linear functionals.

In general it can be shown that for $\Lambda=\Lambda^{\mathrm{all}}$ the $n$th optimal
algorithm $A_{d,n}^*$ is again given by \link{linalgo} and \link{optalgo}.
To see this, we define the Gaussian measure $v_d := \mu_d \circ (\widetilde{\mathrm{App}}_d)^{-1}$ on $L_2([0,1]^d)$.
Then the corresponding covariance operator $C_{v_d}\colon L_2([0,1]^d)\nach L_2([0,1]^d)$ is given by
\begin{gather*}
		C_{v_d} f = \int_{[0,1]^d} K(\,\cdot\,,\bsx) f(\bsx) \, \rd\bsx,
\end{gather*}
which formally equals the definition of $W_d$ above, see \cite[Formula~(16)]{IntandApp}.
Consequently, its eigenpairs $\{(\lambda_{d,m},\eta_{d,m}) \sep m\in\N\}$ are known to be the same as in the worst case setting and it can be shown that the $n$th minimal error satisfies
\begin{gather*}
		e_{\mathrm{avg}}(n; \widetilde{\mathrm{App}}_d,\Lambda^{\mathrm{all}}) 
		= e^{\mathrm{avg}}(A_{d,n}^*; \widetilde{\mathrm{App}}_d)
		= \left( \sum_{i=n+1}^\infty \lambda_{d,i}\right)^{1/2} \qquad \text{for all} \qquad n\in\N_0.
\end{gather*}

\subsection{Quadrature rules based on average case approximation algorithms}\label{subsect:quad}
We are ready to return to the integration problems, the main focus of this article.
Given the space $B_d$ as above, let us define 
\begin{gather*}
		\widetilde{\mathrm{Int}}_d \colon B_d \nach \R, \qquad f \mapsto \widetilde{\mathrm{Int}}_df := \int_{[0,1]^d} f(\bsx) \, \rd\bsx,
\end{gather*}
and let $\mathrm{Int}_d:=\widetilde{\mathrm{Int}}_d\big|_{H_d}$ denote its restriction to the RKHS $H_d\subset B_d$ we are actually interested in.
As before, we approximate this integral by some 
\tr{cubature} rule $Q_{d,n}$ given by \link{QMC}.
Then the average case error of such an integration scheme $Q_{d,n}$ on $B_d$ is defined by
\begin{gather*}
		e^{\mathrm{avg}}(Q_{d,n}; \widetilde{\mathrm{Int}}_d) 
		\tr{:=} \left( \int_{B_d} \abs{\widetilde{\mathrm{Int}}_d f - Q_{d,n}f}^2 \, \rd\mu_d(f) \right)^{1/2}.
\end{gather*}
Now \cite[Corollary 13.1]{NW10} shows that this quantity is exactly equal to the worst case (integration) error for $Q_{d,n}$ on $H_d$ which is given by \link{eq:e2-K}\tr{;} i.e.,
\begin{equation}\label{eq:equal_errors}
	e^{\mathrm{wor}}(Q_{d,n}; H_d) 
	= e^{\mathrm{wor}}(Q_{d,n}; \mathrm{Int}_d)
	= e^{\mathrm{avg}}(Q_{d,n}; \widetilde{\mathrm{Int}}_d).
\end{equation}

Keeping the latter relation in mind, our final goal in this section is to construct a suitable quadrature rule using the following procedure: Given any algorithm $A_{d,n}$ that 
uses at most $n\in\N_0$ function values (information from the class $\Lambda^{\mathrm{std}}$) to approximate $\widetilde{\mathrm{App}}_d$, and a set $\P_{d,r}^\mathrm{int}=\{\bm{t}^{(1)},\ldots,\bm{t}^{(r)}\}$ of $r\in\N$ points in $[0,1]^d$, we let
\begin{gather}\label{def:Q}
	Q_{d,n+r}(f) 
	\tr{:=} Q_{d,n+r}(f;\P_{d,r}^\mathrm{int}, A_{d,n}) 
	\tr{:=} \int_{[0,1]^d} (A_{d,n}f)(\bsx)\,\rd\bsx + \frac{1}{r} \sum_{\ell=1}^r \left( f(\bm{t}^{(\ell)}) - (A_{d,n}f)(\bm{t}^{(\ell)})\right).
\end{gather}
Note that this algorithm $Q_{d,n+r}$ denotes a \tr{cubature} rule\footnote{Indeed $Q_{d,n+r}$ is of the form \link{QMC} provided that $A_{d,n}$ is linear.} for $\widetilde{\mathrm{Int}}_d$ (and $\mathrm{Int}_d$, respectively) which uses no more than $n+r$ function evaluations.
The basic idea behind this construction is a form of variance reduction and can already be found in \cite{W94}; for a clever choice of $A_{d,n}$, the maximum ``energy'' of $f$ is captured by the approximation $A_{d,n}f$ (which is assumed to be integrated exactly) such that the remaining part $f-A_{d,n}f$ can be treated efficiently by a simple QMC method using $r$ additional nodes.
Therefore, it is obvious that the corresponding integration error caused by $Q_{d,n+r}$ will highly depend on the approximation properties of the underlying algorithm $A_{d,n}$ and the quality of the chosen point set $\P_{d,r}^\mathrm{int}$.

A first step towards the desired error bound for $Q_{d,n+r}$ is given by the following estimate which resembles a bound given in \cite[Formula (21)]{IntandApp}. 
It states that, given $A_{d,n}$, a suitable point set $\P_{d,r}^\mathrm{int}$ can always be found.
For the sake of completeness, the (non-constructive) proof based on the proof of \cite[Theorem~3]{W94} is included in the appendix.
\begin{prop}\label{lem:quadrature}
		Let $d,r\in\N$. 
		Then for any algorithm $A_{d,n}$\tr{,} there exists a point set $\P_{d,r}^\mathrm{int}=\P_{d,r}^{\mathrm{int},*}$ such that $Q_{d,n+r}$ as defined in \link{def:Q} satisfies
		\begin{gather}\label{bound:quad_error}
				e^{\mathrm{avg}}(Q_{d,n+r}; \widetilde{\mathrm{Int}}_d)^2
				\leq \frac{1}{r} \, e^{\mathrm{avg}}(A_{d,n}; \widetilde{\mathrm{App}}_d)^2.
		\end{gather}
\end{prop}
\begin{remark}
Of course, the sets $\P_{d,r}^{\mathrm{int},*}$ are not uniquely defined and it may be hard to find these sets in practice.
On the other hand, we can argue that although these bounds are non-constructive, it is known that slightly larger bounds can be achieved with high probability by any random set of points, see, e.g., Plaskota et al. \cite[Remark 2]{PWZ09}, and hence we claim that a suitable set can be found semi-constructively, provided $A_{d,n}$ is given.
\hfill$\square$
\end{remark}

In view of \autoref{lem:quadrature} and our construction \link{def:Q} we are left with finding suitable algorithms $A_{d,n}$ (based on at most $n$ function evaluations) which yield a small average case $L_2$-approximation error.
This can be done inductively.
For this purpose, observe that for each $m\in\N$
\begin{gather*}
		u_m := \frac{1}{m} \sum_{j=1}^m \xi_{d,j}^2,
		\qquad \text{where} \qquad
		\xi_{d,j} := \frac{\eta_{d,j}}{\sqrt{\lambda_{d,j}}} \quad \text{for} \quad j\in\N,
\end{gather*}
defines a probability density on $L_2([0,1]^2)$.
Here the $\xi_{d,j}$'s denote the $L_2$-normalized eigenfunctions~$\eta_{d,j}$ of $W_d$ and $C_{v_d}$, respectively, as described in \autoref{subsect:approx}.
Given $m\in\N$, and an algorithm $A_{d,s}$ that uses $s\in\N_0$ function values to 
approximate $\widetilde{\mathrm{App}}_d$, and a set $\P_{d,q}^{\mathrm{app}}=\{\bm{t}^{(1)},\ldots,\bm{t}^{(q)}\}$ of $q\in\N$ points in $[0,1]^d$,
let
\begin{align}
		A_{d,s+q}(f) 
		&= A_{d,s+q}(f; \P_{d,q}^{\mathrm{app}}, A_{d,s}, m) \nonumber \\
		&= \sum_{j=1}^m \left( \distr{A_{d,s}f}{\xi_{d,j}}_{L_2}+ \frac{1}{q} \sum_{\ell=1}^q \left[ f(\bm{t}^{(\ell)})-(A_{d,s} f)(\bm{t}^{(\ell)}) \right] \frac{\xi_{d,j}(\bm{t}^{(\ell)})}{u_m(\bm{t}^{(\ell)})} \right) \xi_{d,j}, \qquad f\in B_d, \label{def:A}
\end{align}
define another $L_2$-approximation algorithm on $B_d$ that uses $s+q$ evaluations of $f$.
Here we adopt the convention that $0/0:=0$.
Without going into details we want to mention that $A_{d,s+q}$ basically approximates 
the $m$th optimal algorithm $A_{d,m}^*$ (w.r.t.\ $\Lambda^{\mathrm{all}}$) for $\widetilde{\mathrm{App}}_d$; see
\cite[Section~24.3]{NW12} for details.
For $A_{d,s+q}$ defined this way, we have the following error bound which can be found in \cite[Theorem~24.3]{NW12}.
\begin{prop}\label{lem:bootstrap}
		Let $d,m,q\in\N$ be fixed. Then for any algorithm $A_{d,s}$\tr{,} there exists a point set $\P_{d,q}^{\mathrm{app}}=\P_{d,q}^{\mathrm{app},*}$ such that 
		$A_{d,s+q}$ as defined in \link{def:A} 
		fulfills
		\begin{gather}\label{bound:bootstrap}
				e^{\mathrm{avg}}(A_{d,s+q}; \widetilde{\mathrm{App}}_d)^2
				\leq e_{\mathrm{avg}}(m; \widetilde{\mathrm{App}}_d, \Lambda^{\mathrm{all}})^2 + \frac{m}{q} \, e^{\mathrm{avg}}(A_{d,s}; \widetilde{\mathrm{App}}_d)^2.
		\end{gather}
\end{prop}

\begin{remark}
Again the proof of \autoref{lem:bootstrap} is non-constructive since it involves an averaging argument to find a suitable point set $\P_{d,q}^{\mathrm{app},*}$. However, once more, such a point set can be found semi-constructively at the expense of a small additional constant in the error bound above.\hfill$\square$
\end{remark}

Hence, to construct an approximation algorithm $A_{d,n}$ (as required for our cubature rule \link{def:Q}) inductively, in every step we need to choose \tr{positive integers} $m$ and $q$ such that the right-hand side of \link{bound:bootstrap} is minimized.
Of course, this requires some knowledge about the $m$th minimal average case approximation error $e_{\mathrm{avg}}(m; \widetilde{\mathrm{App}}_d, \Lambda^{\mathrm{all}})$ which will be provided by the subsequent assumption.

\begin{assumption}\label{ass:decay}
For $d\in\N$\tr{,} there exist constants $C_d>0$ and $p_d>0$ such that for the ordered sequence of eigenvalues of $W_d = \mathrm{App}_d^{\dagger}\mathrm{App}_d$ (see \autoref{subsect:approx}), the following holds
\begin{gather}\label{Cond}
		e_{\mathrm{avg}}(m; \widetilde{\mathrm{App}}_d, \Lambda^{\mathrm{all}})^2 
		= \sum_{j=m+1}^\infty \lambda_{d,j}
		\leq \frac{C_d}{(m+1)^{p_d}}
		\quad \text{for each} \quad m\in\N_0.
\end{gather}
\end{assumption}
We note in passing that the bound for $m=0$ shows that $C_d$ needs to be larger than $M_{2,d}$. On the other hand, in concrete examples $C_d$ will not be too large, as we will see in the remarks after \autoref{thm:general_quad} and in \autoref{subsect:korobov} below.

Based on \autoref{lem:bootstrap} and \autoref{ass:decay} we now construct a sequence $(A_d^{(k)})_{k\in\N_0}$ of algorithms which perform well for average case approximation $\widetilde{\mathrm{App}}_d$ on $B_d$. 
Afterwards \autoref{lem:quadrature} will be employed to derive the  existence of suitable error bounds for the numerical integration on the RKHS $H_d$.

For $p>0$\tr{,} let $\omega(y):=y+1/y^{p}$ denote a function on the real halfline $y>0$. Then
\begin{align}
	\omega(y)^{1+1/p} 
	&= \left( 1+ \frac{1}{y^{p+1}} \right) y^{1+1/p} \left( 1+ \frac{1}{y^{p+1}} \right)^{1/p}
	= \left( 1+ \frac{1}{y^{p+1}} \right) \left( y^{p+1} +1 \right)^{1/p}\label{bound:omega}\\
	&> 1+ \frac{1}{y^{p+1}} \label{bound:omega2}
\end{align}
shows that $\omega(y)>1$ and thus $2^{p+1} \omega(y)^{1+1/p} > 2$
for all $y,p>0$.
In particular this holds for the minimizer $y_p:=p^{1/(p+1)}$ of $\omega$. Now let
\begin{align*}
	\mathcal{K}_p 
	&:= \max\left\{ k\in\N_0 \sep 2^k \leq 2^{p+1} \omega(y_p)^{1+1/p} \right\} 
	= \floor{ \log_2\! \left( 2^{p+1} \omega(y_p)^{1+1/p} \right)}\in\N, \qquad p>0,
\end{align*}
and let $d\in\N$ be fixed.
Then our sequence of algorithms $(A_d^{(k)})_{k\in\N_0}$ on $B_d$ is defined by
\begin{equation}\label{def:Ak}
	A_d^{(k)} := \begin{cases}
		0 & \quad \text{if} \quad k \leq \mathcal{K}_{p_d}, \\
		A_{d,2^k}(\,\cdot\,;\P_{d,2^{k-1}}^{\mathrm{app},*}, A_d^{(k-1)}, m_k)	 & \quad \text{if} \quad k>\mathcal{K}_{p_d},
	\end{cases}
\end{equation}
where we set
\begin{equation}\label{def:m_k}
	m_k 
	:= \floor{ \left( \frac{C_d \, 2^{k}}{e^{\mathrm{avg}}(A_d^{(k)}; \widetilde{\mathrm{App}}_d)^2} \right)^{1/(p_d+1)} y_{p_d}}, 
	\qquad k \geq \mathcal{K}_{p_d},
\end{equation}
with $p_d$ and $C_d$ taken from \autoref{ass:decay} and $A_{d,2^k}$ as given by \autoref{lem:bootstrap} for $s=q=2^{k-1}$.
The proof of the following lemma can be found in the appendix. 

\begin{lemma}\label{lem:app_sequence}
	Let \autoref{ass:decay} be satisfied.
	Then the sequence $(A_d^{(k)})_{k\in\N_0}$ given by \link{def:Ak} is well-defined. Moreover, every $A_d^{(k)}$, $k\in\N_0$, uses at most $2^k$ points and satisfies the estimate
	\begin{equation}\label{bound:induction}
		e^{\mathrm{avg}}(A_d^{(k)}; \widetilde{\mathrm{App}}_d)^2 
		\leq c(p_d) \, C_d \, (2^k)^{-p_d},
	\end{equation}
where $c(p_d):=2^{p_d(p_d+1)} (1+p_d) \left(1+1/p_d\right)^{p_d}>1$.
\end{lemma}
\begin{remark}
Although the \tr{factor} $c(p_d)$ in \autoref{lem:app_sequence} is super exponential in $p_d$ it might not be too large in concrete applications. E.g., in case of the Monte Carlo rate of convergence ($p_d=1$) it is easily shown that $c(1)=16$.  Moreover, in this case the first non-trivial algorithm from our sequence is given by $A_d^{(\mathcal{K}_{1}+1)}=A_d^{(5)}$ and uses no more than 16 sample points. For $p_d=2$\tr{,} we would have $c(2)= 432$ and again $\mathcal{K}_2=\floor{\log_2(12 \sqrt 3)}=4$.
However, remember that we need to impose stronger decay conditions in \autoref{ass:decay} in order to conclude a higher rate~$p_d$ in \autoref{lem:app_sequence}.
\hfill$\square$
\end{remark}

We are ready to define and analyze our final quadrature rule $Q_{d,N}$ for integration on $H_d$. For this purpose, let $d\in\N$, as well as $N \geq 2$, and suppose that \autoref{ass:decay} is satisfied. 
Setting $\kappa := \floor{\log_2 N} - 1 \in \N_0$ the rule $Q_{d,N}$ is given by
\begin{equation}\label{def:QN}
	Q_{d,N}
	:= Q_{d,2^{\kappa+1}}\left(\,\cdot \,;\P_{d,2^{\kappa}}^{\mathrm{int},*}, A^{(\kappa)}_d \right), 	
\end{equation}
where the algorithm $A^{(\kappa)}_d=A_{d,n}$ is taken out of the sequence $(A_d^{(k)})_{k\in\N_0}$ considered in \autoref{lem:app_sequence} and $Q_{d,2^{\kappa+1}}$ is the quadrature rule from \autoref{lem:quadrature} (with $n=r=2^{\kappa}$).
From the previous considerations it is clear that $Q_{d,N}$ takes the form \link{QMC} with integration nodes $\bm{t}^{(j)}\in[0,1]^d$ from the set
\begin{equation}\label{def:PN}
	\P_{d,N} := \P_{d,2^{\kappa}}^{\mathrm{int},*} \cup \bigcup_{k=\mathcal{K}_{p_d}}^{\kappa-1} \P_{d,2^{k}}^{\mathrm{app},*}
\end{equation}
which can be found semi-constructively. Moreover, the subsequent error bound can be deduced by a straightforward computation.

\begin{theorem}\label{thm:general_quad}
Let $d\in\N$, as well as $N \geq 2$, and suppose that the RKHS $H_d$ satisfies \autoref{ass:decay}. 
Then the \tr{cubature} rule $Q_{d,N}$ defined by \link{def:QN} uses at most $N$ integration nodes and satisfies the bound
\begin{equation*}
	e^{\mathrm{wor}}(Q_{d,N}; H_d)^2 
	\leq 2^{(p_d+2)(p_d+1)} (1+p_d) \left(1+\frac{1}{p_d}\right)^{p_d} \, C_d \, \frac{1}{N^{p_d+1}}.
\end{equation*}
\end{theorem}
\begin{proof}
Obviously, the set $\P_{d,N}$ defined in \link{def:PN} satisfies $\#\P_{d,N} \leq {2^\kappa} +\sum_{k=0}^{\kappa-1} 2^{k} = 2^{\kappa+1}$. From $2^{\kappa+1} \leq N \leq 2^{\kappa+2}$ it thus follows that $\#\P_{d,N} \leq N$.
In addition, \link{eq:equal_errors} together with \link{bound:quad_error} and \link{bound:induction} yields
\begin{align*}
	e^{\mathrm{wor}}(Q_{d,N}; H_d)^2 
	&= e^{\mathrm{avg}}\left(Q_{d,2^{\kappa+1}}(\,\cdot \,;\P_{d,2^{\kappa}}^{\mathrm{int},*}, A^{(\kappa)}_d ); \widetilde{\mathrm{Int}}_d\right)^2 \\
	&\leq \frac{1}{2^{\kappa}} \, e^{\mathrm{avg}}(A^{(\kappa)}_d; \widetilde{\mathrm{App}}_d)^2 \\
	&\leq \frac{1}{2^{\kappa}} \, c(p_d) \, C_d \cdot (2^\kappa)^{-p_d} 
	= 2^{2(p_d+1)} \, c(p_d) \, C_d \, \frac{1}{(2^{\kappa+2})^{p_d+1}},
\end{align*}
with $c(p_d)$ as defined in \autoref{lem:app_sequence}. 
\end{proof}

Finally, let us recall a standard estimate which ensures the validity of \autoref{ass:decay}.
For this purpose, suppose that $\sum_{j=1}^\infty \lambda_{d,j}^{1/\tau}$ is finite for some $\tau > 1$. Then for $k\in\N$\tr{,} the non-increasing ordering of $(\lambda_{d,j})_{j\in\N}$ yields
\begin{gather*}
		k \, \lambda_{d,k}^{1/\tau} 
		\leq \sum_{j=1}^k \lambda_{d,j}^{1/\tau}
		\leq \sum_{j=1}^\infty \lambda_{d,j}^{1/\tau}\tr{;}
		\qquad \text{i.e.,} \qquad
		\lambda_{d,k} \leq k^{-\tau} \left( \sum_{j=1}^\infty \lambda_{d,j}^{1/\tau} \right)^{\tau}.
\end{gather*}
Consequently, for all $m\in\N$\tr{,} we obtain
\begin{align*}
		\sum_{j=m+1}^\infty \lambda_{d,j} 
		&\leq \left( \sum_{j=1}^\infty \lambda_{d,j}^{1/\tau} \right)^{\tau} \! \sum_{j=m+1}^\infty j^{-\tau}
		\le \left( \sum_{j=1}^\infty \lambda_{d,j}^{1/\tau} \right)^{\tau} \frac{m^{1-\tau}}{\tau-1}
		\leq \frac{2^{\tau-1}}{\tau-1} \left( \sum_{j=1}^\infty \lambda_{d,j}^{1/\tau} \right)^{\tau} \frac{1}{(m+1)^{\tau-1}}.
\end{align*}
The case $m=0$ can be handled using Jensen's inequality (\autoref{lem:jensens}) and the fact that $2^y/y\geq 1$ for $y>0$:
\begin{gather*}
		\sum_{j=1}^\infty \lambda_{d,j} 
		\leq \left( \sum_{j=1}^\infty \lambda_{d,j}^{1/\tau} \right)^{\tau} 
		\leq \frac{2^{\tau-1}}{\tau-1} \left( \sum_{j=1}^\infty \lambda_{d,j}^{1/\tau} \right)^{\tau}.
\end{gather*}
Thus, \link{Cond} in \autoref{ass:decay} is satisfied with
\begin{gather*}
		p_d = \tau-1
		\quad \text{and} \quad
		C_d = \frac{2^{\tau-1}}{\tau-1} \left( \sum_{j=1}^\infty \lambda_{d,j}^{1/\tau} \right)^{\tau} 
		\quad \text{for} \quad
		\tau > 1,
\end{gather*}
whenever $\sum_{j=1}^\infty \lambda_{d,j}^{1/\tau}$ converges.

\begin{corollary}\label{cor:quad}
For $d\in\N$\tr{,} let $H_d\hookrightarrow L_2([0,1]^d)$ denote a RKHS and assume that there exists $\tau>1$ such that the ordered sequence of eigenvalues of $W_d = \mathrm{App}_d^{\dagger}\mathrm{App}_d$ satisfies $\sum_{j=1}^\infty \lambda_{d,j}^{1/\tau}<\infty$.
Then for the cubature rule $Q_{d,N}$ considered above it holds
\begin{equation*}
	e^{\mathrm{wor}}(Q_{d,N}; H_d)^2 
	\leq c'(\tau) \left( \sum_{j=1}^\infty \lambda_{d,j}^{1/\tau} \right)^{\tau} N^{-\tau} \qquad \text{for all} \qquad N\geq 2,
\end{equation*}
where $c'(\tau):=2^{\tau(\tau^2-1)} \left( \frac{\tau}{\tau-1}\right)^{\tau}$.
\end{corollary}

\begin{remark}\label{rem:tensor}
Note that the \tr{factor} $c'(\tau)$ in the latter bound deteriorates, as $\tau$ tends to one or to infinity. However, it can be seen numerically that there exists a range for $\tau$ such that $c'(\tau)$ is reasonably small. E.g., for $\tau\in[1.003, 2.04]$\tr{,} we have $c'(\tau)\leq 350$.
Moreover, observe that in general $\tau$ might depend on $d$, whereas for the special case of $d$-fold tensor product spaces $H_d=\bigotimes_{\ell=1}^d H_1$ it can be chosen independent of $d$. On the other hand, in this case
\begin{equation*}
	\{\lambda_{d,j} \,|\, j\in\N\} = \left\{\prod_{\ell=1}^d \lambda_{k_\ell} \sep \bm{k}=(k_1,\ldots,k_d)\in\N^d\right\}
\end{equation*} 
implies $(\sum_{j=1}^\infty \lambda_{d,j}^{1/\tau})^\tau = ( \sum_{k=1}^\infty \lambda_k^{1/\tau} )^{\tau d} \geq (\sum_{k=1}^\infty \lambda_k)^d$ which grows exponentially with the dimension, provided that the underlying space $H_1$ is chosen such that $L_2$-approximation is non-trivial and well-scaled (i.e., if $e_{\mathrm{wor}}(0;\mathrm{App}_1, \Lambda^{\mathrm{all}})^2=1=\lambda_1 \geq \lambda_2>0$).
\hfill$\square$
\end{remark}

\subsection{Application to spaces of permutation-invariant functions}\label{subsect:korobov}
We illustrate the assertions obtained in the preceding subsection by applying them to the reproducing kernel Hilbert spaces $\SI_{I_d}(F_d(r_{\alpha,\bm{\beta}}))$ defined in \autoref{sect:setting}. 
For this purpose we first need to determine the sequence of eigenvalues $(\lambda_{d,j})_{j\in\N}$ of the $d$-variate operator $W_d=W_d(K_{d,I_d})$ given by
\begin{gather*}
		W_d f = (\mathrm{App}_d^{\dagger}\mathrm{App}_d)(f) = \int_{[0,1]^d} f(\bsx)\, K_{d,I_d}(\,\cdot\,,\bsx)\, \rd\bsx,
\end{gather*}
where $\mathrm{App}_d \colon \SI_{I_d}(F_d(r_{\alpha,\bm{\beta}})) \to L_2([0,1]^d)$ is the canonical embedding and $K_{d,I_d}$ denotes the reproducing kernel of the $I_d$-permutation-invariant space $\SI_{I_d}(F_d(r_{\alpha,\bm{\beta}}))$ given in \link{sym_kernel_2}.
Afterwards, we can make use of \autoref{cor:quad} for all $\tau>1$ for which $\sum_{j=1}^\infty \lambda_{d,j}^{1/\tau}$ is finite.

If $\#I_d<2$, then $\SI_{I_d}(F_d(r_{\alpha,\bm{\beta}}))=F_d(r_{\alpha,\bm{\beta}})$ is the $d$-fold tensor product of $F_1(r_{\alpha,\bm{\beta}})$ and thus it suffices to find the univariate eigenvalues $\lambda_{k}=\lambda_{1,k}$, $k\in\N$, see \autoref{rem:tensor}.
From \cite[p.184]{NW08} it \tr{follows} that these eigenvalues are given by
\begin{align}
	\{\lambda_k \sep k\in\N \} 
	&= \{r_{\alpha,\bm{\beta}}(h)^{-1} \sep h\in\Z\} \label{def:univariate}\\
	&= \left\{ \beta_0, \frac{\beta_1}{R(1)^{2\alpha}}, \frac{\beta_1}{R(1)^{2\alpha}}, \frac{\beta_1}{R(2)^{2\alpha}}, \frac{\beta_1}{R(2)^{2\alpha}},\ldots, \frac{\beta_1}{R(j)^{2\alpha}},\frac{\beta_1}{R(j)^{2\alpha}}, \ldots \right\}\tr{;} \nonumber
\end{align}
i.e., we have one eigenvalue $\beta_0$ of multiplicity one and a sequence of distinct eigenvalues $\beta_1/R(m)^{2\alpha}$, $m\in\N$, of multiplicity two.
If we assume that $\beta_0\geq \beta_1/R(m)^{2\alpha}$ for all $m\in\N$, then the latter list is ordered properly according to our needs.
Consequently,
\begin{equation*}
	\sum_{k=1}^\infty \lambda_k^{1/\tau} 
	= \beta_0^{1/\tau} + 2 \, \beta_1^{1/\tau} \sum_{m=1}^\infty R(m)^{-2\alpha/\tau},
	\qquad \tau > 1,
\end{equation*}
is finite if and only if the latter sum converges. Due to our assumptions on $R$ it is easily seen that
\begin{equation}\label{est:sumR}
	\zeta(2\alpha/\tau) \, \left( \frac{1}{R(1)} \right)^{2\alpha/\tau} 
	\leq \sum_{m=1}^\infty R(m)^{-2\alpha/\tau} 
	\leq \zeta(2\alpha/\tau) \, \left( \frac{c_R}{R(1)} \right)^{2\alpha/\tau},
\end{equation}
where $\zeta$ denotes Riemann's zeta function. 
Therefore, $\sum_{k=1}^\infty \lambda_k^{1/\tau}<\infty$ if and only if $1<\tau<2\alpha$.
In addition, from \autoref{rem:tensor} and \link{C_dlambda} we infer that
\begin{align*}
	\left(\sum_{j=1}^\infty \lambda_{d,j}^{1/\tau}\right)^\tau 
	= \left( \sum_{k=1}^\infty \lambda_k^{1/\tau} \right)^{\tau d}
	= \left( \sum_{h\in \Z} r^{-1/\tau}_{\alpha,\bm{\beta}}(h) \right)^{\tau d}
	= \left( \sum_{\bm{h}\in \Z^d} r^{-1/\tau}_{\alpha,\bm{\beta}}(\bm{h}) \right)^{\tau}
	\geq C_{d,\tau}(r_{\alpha,\bm{\beta}})
\end{align*}
(note that $\M_d(\bsh)! = 1 = \# \S_d$ for all $\bsh\in\Z^d$ since $\#I_d<2$)
which shows that in this case the new error bound from \autoref{cor:quad} is worse than the estimate known from \autoref{prop:existence}.

If $\# I_d \geq 2$, then $\SI_{I_d}(F_d(r_{\alpha,\bm{\beta}}))$ is a strict subspace of the tensor product space $F_d(r_{\alpha,\bm{\beta}})$. 
However, in \cite{W12} it has been shown that still there is a relation of the multivariate eigenvalues $\lambda_{d,j}$ with the univariate sequence given in \link{def:univariate}.
In fact, for all $\emptyset \neq I_d=\{i_1,\ldots,i_{\# I_d}\} \subseteq \{1,\ldots,d\}$ it holds
\begin{gather*}
		\{ \lambda_{d,j}  \, | \, j\in\N \} 
		= \left\{ \prod_{\ell=1}^d \lambda_{k_\ell}  \sep \bsk=(k_1,\ldots,k_d)\in\N^d \text{ with } k_{i_1}\leq k_{i_2}\leq\ldots\leq k_{i_{\# I_d}} \right\}
\end{gather*}
and hence the quantity of interest can be decomposed as
\begin{gather}\label{eq:split}
	\sum_{j=1}^\infty \lambda_{d,j}^{1/\tau} 
	= \left[ \sum_{k=1}^\infty \lambda_{k}^{1/\tau} \right]^{d-\# I_d} \ns\,\ns\sum_{\substack{\bsk\in\N^{\# I_d},\\k_1\leq\ldots\leq k_{\# I_d} }} \prod_{\ell=1}^{\# I_d} \lambda_{k_\ell}^{1/\tau}.
\end{gather}
Moreover, it can be shown that this expression is finite if and only if $\sum_{k=1}^\infty \lambda_{k}^{1/\tau}<\infty$ which holds (as before) if and only if $1<\tau < 2\alpha$.

In order to exploit the decomposition \link{eq:split} for \autoref{cor:quad} we proceed \tr{similarly} to the derivation of \cite[Proposition~3.1]{NSW14}. That is, we bound the second factor with the help of some technical lemma (see \autoref{lemmaBound} in the appendix below) and obtain that for all $U\in\N_0$ it holds
\begin{align}
	\sum_{\substack{\bsk\in\N^{\# I_d},\\k_1\leq\ldots\leq k_{\# I_d} }} \prod_{\ell=1}^{\# I_d} \lambda_{k_\ell}^{1/\tau}
	&\leq \lambda_1^{\#I_d/\tau} (\#I_d)^{2\,U} \left( 1 + 2\,U + \sum_{L=1}^{\#I_d} \sum_{ \substack{ \bm{j^{(L)}}\in\N^L,\\2(U+1)\leq j_1^{(L)}\leq\ldots\leq j_L^{(L)} } } \prod_{\ell=1}^L \left( \frac{\lambda_{j^{(L)}_{\ell}}}{\lambda_1}\right)^{1/\tau} \right) \nonumber \\
	&\leq \lambda_1^{\#I_d/\tau} (\#I_d)^{2\,U} \left( 2\,U + \sum_{L=0}^{\#I_d} \left[ \sum_{ j=2(U+1)}^{\infty}  \left( \frac{\lambda_{j}}{\lambda_1}\right)^{1/\tau} \right]^{L} \right). \label{est:sym_sum}
\end{align}
Given $1<\tau<2\alpha$ let us define
\begin{equation*}
	\varrho_\tau(U) 
	:= \sum_{j=2(U+1)}^\infty \left(\frac{\lambda_{j}}{\lambda_1} \right)^{1/\tau} 
	= 2 \left(\frac{\beta_1}{\beta_0} \right)^{1/\tau} \sum_{m=U+1}^{\infty} R(m)^{-2\alpha/\tau}
	\qquad \text{for every} \qquad U\in\N_0.
\end{equation*}
Then the finiteness of \link{est:sumR} implies that there necessarily exists some $U_\tau^*:=U_\tau^*(R,\alpha,\bm{\beta},\tau)\in\N_0$ with
\begin{equation*}
	\varrho_\tau^* = \varrho_\tau(U_\tau^*) < 1.
\end{equation*}
In conjunction with \link{eq:split} and \link{est:sym_sum} this finally yields that $\sum_{j=1}^\infty \lambda_{d,j}^{1/\tau}$ can be upper bounded by
\begin{align*}
	&\lambda_1^{(d-\#I_d)/\tau} \left[ \sum_{k=1}^\infty \left( \frac{\lambda_{m}}{\lambda_1}\right)^{1/\tau} \right]^{d-\# I_d} \lambda_1^{\#I_d/\tau} (\#I_d)^{2\,U_\tau^*} \left( 2\,U_\tau^* + \frac{1}{1-\varrho_\tau^*} \right) \\
	&\;= \beta_0^{d/\tau} \left[ 1 + 2 \left( \frac{\beta_1}{\beta_0} \right)^{1/\tau} \sum_{m=1}^\infty R(m)^{-2\alpha/\tau} \right]^{d-\# I_d} \left( 2\,U_\tau^* + \frac{1}{1-\varrho_\tau^*} \right) (\#I_d)^{2\,U_\tau^*} \\
	&\;\,\leq e(0,d; \SI_{I_d}(F_d(r_{\alpha,\bm{\beta}})))^{2/\tau} \left[ 1 + 2 \left( \frac{\beta_1 c_R^{2\alpha}}{\beta_0 R(1)^{2\alpha}} \right)^{1/\tau} \zeta(2\alpha/\tau) \right]^{d-\# I_d} \!\! \left( 2\,U_\tau^* + \frac{1}{1-\varrho_\tau^*} \right) (\#I_d)^{2\,U_\tau^*},
\end{align*}
since $\lambda_1^{d}=\beta_0^d=e(0,d; \SI_{I_d}(F_d(r_{\alpha,\bm{\beta}})))^2$.
Thus, \autoref{cor:quad} implies the subsequent result.

\begin{theorem}\label{thm:higherorder}
	Consider the integration problem on the $I_d$-permutation-invariant subspaces $\SI_{I_d}(F_d(r_{\alpha,\bm{\beta}}))$, where $d\geq 2$ and $I_d\subseteq\{1,\ldots,d\}$ with $\# I_d \geq 2$. 
	Moreover, assume that
	\begin{gather}\label{assump_new}
 		\frac{\beta_1}{\beta_0 \, R(m)^{2\alpha}} \leq 1
 		\quad \text{for all} \quad m\in\N.
	\end{gather}
	Then for every $N\geq 2$\tr{,} the worst case error of the \tr{cubature} rule $Q_{d,N}$ defined in \link{def:QN} satisfies
	\begin{align}
		&e^{\mathrm{wor}}(Q_{d,N}; \SI_{I_d}(F_d(r_{\alpha,\bm{\beta}}))) \nonumber\\
		&\qquad\qquad \leq e(0,d; \SI_{I_d}(F_d(r_{\alpha,\bm{\beta}}))) \,\left(2\,U_\tau^* + \frac{1}{1-\varrho_\tau^*}\right)^{\tau/2} 2^{\tau(\tau^2-1)/2} \left( \frac{\tau}{\tau-1}\right)^{\tau/2}   \nonumber \\
			&\qquad\qquad\qquad\qquad\qquad \times \left[ 1 + 2 \left( \frac{\beta_1 c_R^{2\alpha}}{\beta_0 R(1)^{2\alpha}} \right)^{1/\tau} \zeta(2\alpha/\tau) \right]^{(d-\#I_d)\tau/2} (\#I_d)^{\tau \,U_\tau^*} \, N^{-\tau/2} \label{errorBound_new}
	\end{align}
	for all $1<\tau<2 \alpha$ and $U_\tau^*\in\N_0$ such that
	\begin{gather}\label{cond:rho}
		\varrho_\tau^* 
		= 2 \left( \frac{\beta_1}{\beta_0} \right)^{1/\tau} \sum_{m=U_\tau^*+1}^{\infty} R(m)^{-2\alpha/\tau} < 1.
	\end{gather}
\end{theorem}

\begin{remark}
We conclude this section by some final remarks on \autoref{thm:higherorder}.
\begin{itemize}
	\item[(i)] First of all, note that using a sufficiently small (but constant) value of $\beta_1$ the condition~\link{cond:rho} can always be fulfilled with $U_\tau^*=0$.
	\item[(ii)] Observe that the bound \link{errorBound_new} combines the advantages of the general existence result for QMC algorithms (\autoref{thm:tractability}) with the higher rates of convergence in the estimates for lattice rules (as well as their component-by-component construction) stated in \autoref{prop:existence} (and \autoref{thm:cbc}, respectively).
	In fact, \link{errorBound_new} structurally resembles the bound \link{errorBound} from \autoref{thm:tractability};
 besides the initial error and some absolute constants it contains a term with grows exponentially in $d-\#I_d$ (but not in $d$ itself!), as well as a polynomial in $\#I_d$. 
 On the other hand, similar to the lattice rule approach, the Monte Carlo rate of convergence is enhanced by a factor of~$\tau$.
Hence, the construction given in \autoref{subsect:quad} provides a cubature rule which achieves a worst case error of $\0(n^{-\alpha})$ while the implied constants grow at most polynomially with the dimension $d$, provided that we assume a sufficiently large amount of permutation-invariance (i.e., if $d-\#I_d \in \0(\ln d)$). In other words, it can be used to deduce (strong) polynomial tractability.	
	\item[(iii)] We stress that (in contrast to the rank-$1$ lattice rules studied in \autoref{sect:cbc}) $Q_{d,N}$ is a \tr{\emph{general weighted cubature}} rule and its integration nodes do not necessarily belong to some regular structure such as an integration lattice. However, as exposed in \autoref{subsect:quad}, they can be found semi-constructively.
	\item[(iv)] Finally, we mention that the condition \link{assump_new} improves on \link{assump} by a factor of two.
	\hfill$\square$
\end{itemize}
\end{remark}

\color{black}

\begin{appendices}
\section{}
\label{sect:appendix}

In this final section we collect auxiliary results, as well as some technical proofs we skipped in the presentation above.

\subsection{Auxiliary estimates}
For the reader's convenience let us recall a standard estimate which is sometimes referred to as \emph{Jensen's inequality}.
\begin{lemma}\label{lem:jensens}
Let $(a_j)_{j\in\N}$ \tr{be} an arbitrary sequence of non-negative real numbers. Then, for every $0<q\leq p< \infty$,
\begin{equation*}
	\left( \sum_{j=1}^\infty a_j^p \right)^{1/p}
	\leq \left( \sum_{j=1}^\infty a_j^q \right)^{1/q}
\end{equation*}
whenever the right-hand side is finite.
\end{lemma}

In addition, in \autoref{subsect:korobov} we make use of the following technical lemma (with $\sigma_m:=\lambda_{m}^{1/\tau}$ for $m\in\N$ and $s:=\#I_d$) which was employed already in \cite{NSW14} and \cite{W12}. For a detailed proof (of an insignificantly modified version)\tr{,} we refer to \cite[Lemma~5.8]{Wei2015}.

\begin{lemma}\label{lemmaBound}
 	Let $(\sigma_m)_{m\in\N}$ be a sequence of non-negative real numbers with $\sigma_1 = \sup_{m\in\N} \sigma_m > 0$ and set $\sigma_{L,\bm{k}}:=\prod_{\ell=1}^L \sigma_{k_\ell}$ for $\bm{k}\in\N^L$ and $L\in\N$.
 	Then for all $U\in\N_0$ and every $s\in\N$ it holds
 	\begin{gather*}
 		\sum_{\substack{\bm{k}\in\N^s,\\k_1\leq\ldots\leq k_s}} \sigma_{s,\bm{k}}
 		\leq \sigma_1^s \, s^{2\,U} \left( 1 + 2\,U + \sum_{L=1}^s \sigma_1^{-L} \sum_{ \substack{ \bm{j^{(L)}}\in\N^L,\\2(U+1)\leq j_1^{(L)}\leq\ldots\leq j_L^{(L)} } } \sigma_{L,\bm{j^{(L)}}} \right)
 	\end{gather*}
	with equality at least for $U=0$.
\end{lemma}

\subsection{Postponed proofs}
We start with showing \autoref{lem:quadrature} which relates the average case integration error of $Q_{d,n+r}$ with the error of the approximation algorithm $A_{d,n}$ used to construct $Q_{d,n+r}$, see \autoref{subsect:approx}.

\begin{proof}[Proof of \autoref{lem:quadrature}]
For any fixed collection of points $\bm{t}^{(1)},\ldots,\bm{t}^{(r)}\in[0,1]^d$ 
it is easy to see that 
\begin{equation*}
	\widetilde{\mathrm{Int}}_d f - Q_{d,n+r}f
	= \widetilde{\mathrm{Int}}_d(f-A_{d,n}f) - \frac{1}{r} \sum_{\ell=1}^r (f-A_{d,n}f)(\bm{t}^{(\ell)}).
\end{equation*}
Squaring this expression and taking the expectation with respect to $\bm{t}^{(\ell)}$, $\ell=1,\ldots,r$, gives
\begin{gather*}
	\int_{[0,1]^d\times\tr{\cdots}\times[0,1]^d}\abs{\widetilde{\mathrm{Int}}_d f - Q_{d,n+r}f}^2  \rd \bm{t}^{(1)}\ldots \rd \bm{t}^{(r)}
	= \frac{1}{r} \left( \norm{f-A_{d,n}f \sep L_2}^2 - (\widetilde{\mathrm{Int}}_d(f-A_{d,n}f))^2 \right), 
\end{gather*}
which implies (by integrating over $B_d$, interchanging the integrals, and estimating the negative term) that
\begin{gather*}
	\int_{[0,1]^d\times\tr{\cdots}\times[0,1]^d} e^{\mathrm{avg}}(Q_{d,n+r}; \widetilde{\mathrm{Int}}_d)^2 \, \rd \bm{t}^{(1)}\ldots \rd \bm{t}^{(r)}
				\leq \frac{1}{r} \int_{B_d} \norm{f-A_{d,n}f \sep L_2}^2 \, \rd\mu_d(f).
\end{gather*}
Now the integral on the right-hand side is nothing but $e^{\mathrm{avg}}(A_{d,n}; \widetilde{\mathrm{App}}_d)^2$ and
by the usual argument (mean value theorem) there exists  a point set
$\P_{d,r}^{\mathrm{int},*}=\{\bm{t}^{(1)},\ldots,\bm{t}^{(r)}\} \subset [0,1]^d$ such that 
$e^{\mathrm{avg}}(Q_{d,n+r}(\,\cdot\,;\P_{d,r}^{\mathrm{int},*},A_{d,n}); \widetilde{\mathrm{Int}}_d)^2$ is smaller than the average on the left.
\end{proof}

It remains to prove \autoref{lem:app_sequence} which justifies the iterative construction of the sequence of approximation algorithms $(A_d^{(k)})_{k\in\N_0}$.

\begin{proof}[Proof of \autoref{lem:app_sequence}]
For this proof let $p:=p_d$ and $C:=C_d$ denote the constants from \autoref{ass:decay}.
We are going to prove that $A^{(k)}_d$ as defined in \link{def:Ak} satisfies
\begin{equation}\label{bound:e}
	e^{\mathrm{avg}}(A_d^{(k)}; \widetilde{\mathrm{App}}_d)^2 
	\leq 2^{p(p+1)} \, \omega(y_p)^{p+1} \, C \, \frac{1}{(2^k)^{p}} \qquad \text{for all} \qquad k\in\N_0.
\end{equation}
Then \link{bound:omega} together with the definition of $y_p:=p^{1/(p+1)}$ implies the claim.

\emph{Step 1.}
Let us consider $k\in\{0,1,\ldots,\mathcal{K}_p\}$ first. Note that due to the considerations after formula \link{bound:omega} this set is not empty.
Since for $k\leq \mathcal{K}_p$ we have $2^{p(p+1)} \, \omega(y_p)^{p+1} \, (2^k)^{-p} \geq 1$, it suffices to show that $e^{\mathrm{avg}}(A_d^{(k)}; \widetilde{\mathrm{App}}_d)^2 \leq C$ in this case. Since $A^{(k)}_d\equiv 0$, this is true due to \link{def:e_avg} and \link{Cond} applied for $m=0$. The remaining assertions are trivial for $k\leq \mathcal{K}_p$.

\emph{Step 2.}
Observe that $m_k$ as defined in \link{def:m_k} is at least one if and only if
\begin{equation*}
	e^{\mathrm{avg}}(A_d^{(k)}; \widetilde{\mathrm{App}}_d)^2 \leq C \, 2^k \, y_p^{p+1}.
\end{equation*}
Moreover, the right-hand side of \link{bound:e} is less than or equal to $C \, 2^k \, y_p^{p+1}$ if and only if
\begin{equation}\label{cond:k}
	2^{p} \left( 1+ \frac{1}{y_p^{p+1}}\right) \leq 2^k.
\end{equation}
Therefore, every $k\in\N_0$ which satisfies \link{bound:e} and \link{cond:k} fulfills $m_k \geq 1$. 

We stress that the conditions \link{bound:e} and \link{cond:k} hold true at least for $k=\mathcal{K}_p$. In fact, the validity of \link{bound:e} has been shown already in Step 1 and (by the definition of $\mathcal{K}_p$) we have
that
\begin{equation*}
	2^{\mathcal{K}_p} \leq 2^{p+1} \omega(y_p)^{1+1/p} < 2^{\mathcal{K}_p+1},
\end{equation*}
which implies $2^{\mathcal{K}_p}> 2^{p}\, \omega(y_p)^{1+1/p} > 2^{p} (1+ 1/y_p^{p+1})$ using \link{bound:omega2} for the last estimate.
Furthermore, note that with $k=\mathcal{K}_p$ the condition \link{cond:k} holds true for every $k\geq \mathcal{K}_p$.

\emph{Step 3.}
Now we prove \link{bound:e} for $A^{(\mathcal{K}_p+1)}_d$. 
For this purpose, we let $k=\mathcal{K}_p$ and make use of \autoref{lem:bootstrap} with $m=m_k$ and $q=2^{k}$ (here we need that $m_k \geq 1$). Employing the definition of~$m_k$ given in \link{def:m_k} together with \link{Cond} from \autoref{ass:decay} we obtain
\begin{align*}
	&e^{\mathrm{avg}}(A_d^{(k+1)}; \widetilde{\mathrm{App}}_d)^2 \\
	&\qquad \leq e_{\mathrm{avg}}(m_k; \widetilde{\mathrm{App}}_d, \Lambda^{\mathrm{all}})^2 + \frac{m_k}{2^{k}} \, e^{\mathrm{avg}}(A^{(k)}_d; \widetilde{\mathrm{App}}_d)^2 \\
	&\qquad \leq \frac{C^{1/(p+1)} \, e^{\mathrm{avg}}(A^{(k)}_d; \widetilde{\mathrm{App}}_d)^{2\, p/(p+1)}}{2^{k \, p/(p+1)}\, y_p^{p}} + \frac{C^{1/(p+1)} \, e^{\mathrm{avg}}(A^{(k)}_d; \widetilde{\mathrm{App}}_d)^{2\, p/(p+1)}}{2^{k \, p/(p+1)}} \, y_p \\
	&\qquad = C^{1/(p+1)} \, \left[ e^{\mathrm{avg}}(A^{(k)}_d; \widetilde{\mathrm{App}}_d)^{2} \,(2^k)^{p} \right]^{p/(p+1)} \,  \frac{2^{p} \, \omega(y_p)}{(2^{k+1})^{p}}.
\end{align*}
Finally, we plug in the upper estimate \link{bound:e} for the error of $A_d^{(k)}$ and derive the same bound for $A_d^{(k+1)}=A_d^{(\mathcal{K}_p+1)}$. 
But now Step 2 yields that also $m_{\mathcal{K}_p+1} \geq 1$ such that the same reasoning as before implies \link{bound:e} for $k=\mathcal{K}_p+2$ as well, and (by induction) for any further $k$.

In conclusion these arguments show that the tail sequence $(A^{(k)}_d)_{k > \mathcal{K}_p}$ is well-defined and that it satisfies the claimed error bound. Since in the construction we add at most $2^{k-1}$ points when turning from $k-1$ to $k$, each operator $A^{(k)}_d$ obviously uses no more than $2^k$ sample points. Hence, the proof is complete.
\end{proof}
\end{appendices}

\section*{Acknowledgements}
The authors are grateful to the anonymous reviewers who helped in improving the paper.

\bibliographystyle{abbrv} 
\bibliography{biblio}

\begin{thebibliography}{10}

\bibitem{A50}
N.~Aronszajn.
\newblock Theory of reproducing kernels.
\newblock {\em Trans. Amer. Math. Soc.}, 68(3):337--404, 1950.

\bibitem{DKS13}
J.~Dick, F.~Y. Kuo, and I.~H. Sloan.
\newblock High dimensional integration: The quasi-{M}onte {C}arlo way.
\newblock {\em Acta Numer.}, 22:133--288, 2013.

\bibitem{IntandApp}
F.~J. Hickernell and H.~Wo{\'z}niakowski.
\newblock Integration and approximation in arbitrary dimensions.
\newblock {\em Adv.\ Comput.\ Math.}, 12:25--58, 2000.

\bibitem{K03}
F.~Y. Kuo.
\newblock Component-by-component constructions achieve the optimal rate of
  convergence for multivariate integration in weighted {K}orobov and {S}obolev
  spaces.
\newblock {\em J.~Complexity}, 19:301--320, 2003.

\bibitem{JK02}
F.~Y. Kuo and S.~Joe.
\newblock Component-by-component construction of good lattice rules with a
  composite number of points.
\newblock {\em J.~Complexity}, 18:943--976, 2002.

\bibitem{NW08}
E.~Novak and H.~Wo{\'z}niakowski.
\newblock {\em Tractability of {M}ultivariate {P}roblems. {V}ol. {I}: {L}inear
  {I}nformation}, volume~6 of {\em EMS Tracts in Mathematics}.
\newblock European Mathematical Society (EMS), Z\"urich, 2008.

\bibitem{NW10}
E.~Novak and H.~Wo{\'z}niakowski.
\newblock {\em Tractability of {M}ultivariate {P}roblems. {V}ol. {II}:
  {S}tandard {I}nformation for {F}unctionals}, volume~12 of {\em EMS Tracts in
  Mathematics}.
\newblock European Mathematical Society (EMS), Z\"urich, 2010.

\bibitem{NW12}
E.~Novak and H.~Wo{\'z}niakowski.
\newblock {\em Tractability of {M}ultivariate {P}roblems. {V}ol. {III}:
  {S}tandard {I}nformation for {L}inear {O}perators}, volume~18 of {\em EMS
  Tracts in Mathematics}.
\newblock European Mathematical Society (EMS), Z\"urich, 2012.

\bibitem{Nuy2014}
D.~Nuyens.
\newblock The construction of good lattice rules and polynomial lattice rules.
\newblock In {Kritzer, P.}, {Niederreiter, H.}, {Pillichshammer, F.}, and
  {Winterhof, A.}, editors, {\em Uniform Distribution and Quasi-Monte Carlo
  Methods: Discrepancy, Integration and Applications}, volume~15 of {\em Radon
  Series on Computational and Applied Mathematics}, pages 223--256. De Gruyter,
  Berlin, Boston, 2014.

\bibitem{NSW14}
D.~Nuyens, G.~Suryanarayana, and M.~Weimar.
\newblock Rank-1 lattice rules for multivariate integration in spaces of
  permutation-invariant functions: {E}rror bounds and tractability.
\newblock {\em Adv.\ Comput.\ Math.}, 42(1):55--84, 2016.

\bibitem{PWZ09}
L.~Plaskota, G.~Wasilkowski, and Y.~Zhao.
\newblock New averaging technique for approximating weighted integrals.
\newblock {\em J.~Complexity}, 25:268--291, 2009.

\bibitem{SKJ02_1}
I.~H. Sloan, F.~Y. Kuo, and S.~Joe.
\newblock Constructing randomly shifted lattice rules in weighted {S}obolev
  spaces.
\newblock {\em SIAM Journal on Numerical Analysis}, 40:1650--1665, 2002.

\bibitem{SKJ02}
I.~H. Sloan, F.~Y. Kuo, and S.~Joe.
\newblock On the step-by-step construction of quasi-{M}onte {C}arlo integration
  rules that achieve strong tractability error bounds in weighted {S}obolev
  spaces.
\newblock {\em Math. Comp.}, 71:1609--1640, 2002.

\bibitem{SR02}
I.~H. Sloan and A.~V. Reztsov.
\newblock Component-by-component construction of good lattice rules.
\newblock {\em Math. Comp.}, 71:263--273, 2002.

\bibitem{W94}
G.~Wasilkowski.
\newblock Integration and approximation of multivariate functions: Average case
  complexity with isotropic wiener measure.
\newblock {\em J.~Approx. Theory}, 77:212–--227, 1994.

\bibitem{W12}
M.~Weimar.
\newblock The complexity of linear tensor product problems in (anti)symmetric
  {H}ilbert spaces.
\newblock {\em J.~Approx. Theory}, 164(10):1345--1368, 2012.

\bibitem{W14}
M.~Weimar.
\newblock On lower bounds for integration of multivariate permutation-invariant
  functions.
\newblock {\em J.~Complexity}, 30(1):87--97, 2014.

\bibitem{Wei2015}
M.~Weimar.
\newblock Breaking the curse of dimensionality.
\newblock {\em Dissertationes Math.}, 505:1--112, 2015.

\bibitem{Y10}
H.~Yserentant.
\newblock {\em {R}egularity and {A}pproximability of {E}lectronic {W}ave
  {F}unctions}.
\newblock Lecture Notes in Mathematics. Springer-Verlag, Berlin, 2010.

\end{thebibliography}
\end{document}